\documentclass[a4paper,10pt,leqno]{amsart}
\usepackage[total={6in,9in},
top=1in, left=1in, right=1in, bottom=1in]{geometry}
\usepackage{amsmath}
\usepackage[italian, english]{babel}
\usepackage{amsfonts}
\usepackage{amssymb}
\usepackage{dsfont}
\usepackage{mathrsfs}
\usepackage{graphicx}
\usepackage{setspace}
\usepackage{fancyhdr}
\usepackage{amsthm}
\usepackage{empheq}
\usepackage{cases}
\usepackage[all]{xy}
\usepackage{stmaryrd}
\usepackage[colorlinks, citecolor=blue, linkcolor=red]{hyperref}
\usepackage{mathrsfs}
\newcommand{\varrg}{(M, g)}

\newcommand{\erre}{\mathds{R}}

\newcommand{\cinf}{C^{\infty}(M)}

\newcommand{\ricc}{\operatorname{Ric}}
\newcommand{\diver}{\operatorname{div}}
\newcommand{\hess}{\operatorname{Hess}}

\newcommand{\set}[1]{{\left\{#1\right\}}}               
\newcommand{\pa}[1]{{\left(#1\right)}}                  
\newcommand{\sq}[1]{{\left[#1\right]}}                  
\newcommand{\abs}[1]{{\left|#1\right|}}                 

\newcommand{\riemanng}[1]{\pa{#1,g}}                      

\renewcommand{\tilde}[1]{\widetilde{#1}}

\newtheorem{theorem}{\textbf{Theorem}}[section]
\newtheorem{lemma}[theorem]{\textbf{Lemma}}
\newtheorem{proposition}[theorem]{\textbf{Proposition}}

\newtheorem{defi}[theorem]{\textbf{Definition}}
\theoremstyle{remark}
\newtheorem{rem}[theorem]{\textbf{Remark}}

\numberwithin{equation}{section}

\title[CRS]
{Conformal Ricci Solitons and related Integrability Conditions}

\date{\today} 



\begin{document}
\maketitle

\begin{center}
\textsc{\textmd{G. Catino\footnote{Politecnico di Milano, Italy.
Email: giovanni.catino@polimi.it. Supported
by GNAMPA projects ``Equazioni differenziali con invarianze in
analisi globale'' and ``Equazioni di evoluzione geometriche e strutture di tipo Einstein''.}, P.
Mastrolia\footnote{Universit\`{a} degli Studi di Milano, Italy.
Email: paolo.mastrolia@gmail.com. Partially supported by FSE,
Regione Lombardia and GNAMPA project ``Analisi Globale ed Operatori Degeneri''.}, D. D. Monticelli\footnote{Universit\`{a} degli
Studi di Milano, Italy. Email: dario.monticelli@unimi.it. Supported
by GNAMPA projects ``Equazioni differenziali con invarianze in
analisi globale'' and ``Analisi Globale ed Operatori Degeneri''.} and M.
Rigoli\footnote{Universit\`{a} degli Studi di Milano, Italy. Email:
marco.rigoli@unimi.it. \\ The first, the second and the third authors are members of the Gruppo Nazionale per
l'Analisi Matematica, la Probabilit\`{a} e le loro Applicazioni (GNAMPA)
of the Istituto Nazionale di Alta Matematica (INdAM).
}, }}
\end{center}
\begin{abstract}
In this paper we introduce, in the Riemannian setting, the notion of \emph{conformal Ricci soliton}, which includes as particular cases \emph{Einstein manifolds}, \emph{conformal Einstein manifolds} and (generic and gradient) \emph{Ricci solitons}. We provide here some necessary integrability conditions for the existence of these structures that also recover, in the corresponding contexts, those already known in the literature for conformally Einstein manifolds and  for gradient Ricci solitons. A crucial tool in our analysis is the construction of some appropriate and highly nontrivial  $\pa{0,3}$-tensors related to the geometric structures, that in the special case of gradient Ricci solitons become the celebrated tensor $D$ recently introduced by Cao and Chen. A significant part of our investigation, which has independent interest, is the derivation of a number of commutation rules for covariant derivatives (of functions and tensors) and of transformation laws of some geometric objects under a conformal change of the underlying metric.
\end{abstract}


\section{Introduction}

%
%
%
%
%

In recent years the pioneering works of R. Hamilton (\cite{hamilton}) and G. Perelman (\cite{perelman}) towards the solution of the Poincar\'e conjecture in dimension $3$ have produced a flourishing activity in the research of self similar solutions, or \emph{solitons}, of the Ricci flow. The study of the geometry of solitons, in particular their classification in dimension $3$, has been essential in providing a positive answer to the conjecture; however, in higher dimension and in the complete, possibly noncompact case, the understanding of the geometry and the classification of solitons seems to remain a desired goal for a not too proximate future. In the \emph{generic} case a \emph{soliton structure} on the Riemannian manifold $\varrg$ is the choice (if any) of a smooth vector field $X$ on $M$ and a real constant $\lambda$ satisfying the structural requirement
\begin{equation}\label{EQ0.1_RicSolEq}
  \ricc + \frac 12 \,\mathcal{L}_Xg = \lambda g,
\end{equation}
where $\ricc$ is the Ricci tensor of the metric $g$  and $\mathcal{L}_Xg$ is the Lie derivative of this latter in the direction of $X$. In what follows we shall refer to $\lambda$ as to the \emph{soliton constant}. The soliton is called \emph{expanding}, \emph{steady} or \emph{shrinking} if, respectively, $\lambda < 0$, $\lambda = 0$ or $\lambda >0$. When $X$ is the gradient of a potential $f \in C^{\infty}\pa{M}$, the soliton is called a \emph{gradient Ricci soliton} and the previous equation \eqref{EQ0.1_RicSolEq} takes the form
\begin{equation}\label{EQ0.2_GradRicSolEq}
  \ricc + \hess{f} = \lambda g.
\end{equation}
Both equations \eqref{EQ0.1_RicSolEq} and \eqref{EQ0.2_GradRicSolEq} can be considered as perturbations of the Einstein equation
\begin{equation}\label{EQGianniEinstein}
\ricc = \lambda g
\end{equation}
and reduce to this latter in case $X$ or $\nabla f$ are Killing vector fields. When $X=0$ or $f$ is constant we call the underlying Einstein manifold a \emph{trivial} Ricci soliton.
The great interest raised by these structures is also shown by the rapidly increasing number of works devoted to their study; for instance, just to cite a few of them,  we mention in particular  \cite{Ivey}, \cite{perelman1}, \cite{ELnM},
\cite{NiWallach},  \cite{ZHZhang}, \cite{brendle}, \cite{Naber}, \cite{petwylie3},
\cite{CaoZhou}, \cite{HDCaoChen_Steady}, \cite{CatMantEv},
\cite{PRiS}, \cite{XDCaoWangZhang_ShrinkingRS},
\cite{Catino_pinched}, \cite{CaoChen}, \cite{CaoCatinoChenMantMazz},
\cite{MasRigRim}, \cite{CMMRGen} (and references therein) on Ricci solitons and \cite{Jensen}, \cite{Besse}, \cite{WangZiller1}, \cite{WangZiller2}, \cite{MasMonRig_Curvature} (and references therein) on Einstein manifolds.



A natural question, which arises for instance in conformal geometry, is to construct \emph{conformally Einstein manifolds}, i.e. Riemannian manifolds $\varrg$ for which there exists a pointwise conformal deformation $\tilde{g}=e^{2u}g$, $u\in\cinf$, such that the new metric $\tilde{g}$ is Einstein. This problem has received a considerable amount of attention by mathematicians and physicists in the last decades: just to mention some old and recent papers we cite the pioneering work of Brinkmann, \cite{Brink}, Yano and Nagano, \cite{YanoNagano}, Gover and Nurowski, \cite{GoverNurowski}, Kapadia and Sparling, \cite{KapadiaSparling}, Derdzisnki and Maschler, \cite{DerdMasch}, and references therein. In particular in \cite{GoverNurowski} the authors describe two necessary integrability conditions for the existence of the conformal deformation $\tilde{g}$ realizing the Einstein metric. They are of course related to the system
\begin{equation}\label{EQ0.3}
  \widetilde\ricc = \lambda \tilde g,
\end{equation}
where tilded  quantities refer to the metric $\tilde g$, and they are expressed in terms of the Cotton, Weyl and Bach tensors and the gradient of $u$ in the background metric $g$ (see Section \ref{SecDef} for precise definitions); precisely, performing a computation in some sense reminiscent of the classical Cartan's approach to the treatment of differential systems, Gover and Nurowski show that if $\varrg$ is a conformally Einstein Riemannian manifold, then the Cotton tensor, the Weyl tensor, the Bach tensor and the exponent $u$ of the stretching factor satisfy the conditions (see also Proposition \ref{Prop_Gover_Nurowski2.1})
 \begin{align}
   &C_{ijk} -(m-2)u_tW_{tijk}=0, \label{GNintro1} \\ &B_{ij}-(m-4)u_tu_kW_{itjk}=0. \label{GNintro2}
 \end{align}

On the other hand, Cao and Chen in \cite{CaoChen0} and \cite{CaoChen} study the geometry of  Bach flat gradient solitons, introducing a $(0, 3)$-tensor $D$ related to the geometry of the level surfaces of the potential $f$ that generates the soliton structure. The vanishing of $D$, obtained \emph{via} the vanishing of the Bach tensor, is a crucial ingredient in their classification of a wide family of complete gradient Ricci solitons;  in particular in their proof they show that every gradient Ricci soliton satisfies the two conditions
 \begin{align}
\label{piddu1}     &C_{ijk}+f_t W_{tijk} = D_{ijk},\\ \label{nonumber}    &B_{ij} = \frac{1}{m-2}\sq{D_{ijk, k}+\pa{\frac{m-3}{m-2}}f_tC_{jit}}.
  \end{align}
  The above equations must be intended as integrability conditions for  solitons, in the same way as \eqref{GNintro1} and \eqref{GNintro2} are related to  conformally Einstein manifolds. We observe that the aforementioned classification result has been recently generalized by the present authors in \cite{CMMRGet} to a new general structure (which includes Ricci solitons, Yamabe solitons, quasi-Einstein manifolds and almost Ricci solitons), called  \emph{(gradient) Einstein-type manifold}, for which the corresponding integrability conditions have also been computed.

 In the present work we introduce for the first time the counterpart of the tensor $D$ in the case of \emph{generic Ricci solitons}: we call it $D^X$ and we show that in this setting the integrability conditions take the form
 \begin{align}
     &C_{ijk}+X_t W_{tijk} = D^X_{ijk}, \label{INTROfirstGenericRSIntCondition}\\     &B_{ij} = \frac{1}{m-2}\pa{D^X_{ijk, k}+\frac{m-3}{m-2}X_tC_{jit}+\frac 12\pa{X_{tk}-X_{kt}}W_{itjk}} \label{INTROsecondGenericRSIntCondition}
  \end{align}
(see Theorem \ref{GenRS_TH_FirstCond}). We explicitly note that, if $X=\nabla f$ for some $f \in \cinf$, then $D^X \equiv D$ and the two previous equations become, as one should expect, \eqref{piddu1} and \eqref{nonumber} respectively.

Since Einstein metrics are trivial solitons, it is now natural to study \emph{conformal Ricci solitons}, i.e. to search for pointwise conformal transformations of the metric as above, such that the manifold $\pa{M, \tilde g}$ is a gradient Ricci soliton, that is for some $f\in \cinf$ and $\lambda \in \erre$ we have the structural relation
\begin{equation}\label{EQ0.4}
  \widetilde \ricc + \widetilde\hess\pa{f} = \lambda\tilde{g}.
\end{equation}

One of the main aims of the paper is to produce integrability conditions corresponding to these structures; in their study we introduce here for the first time a natural $(0,3)$-tensor, which we denote by $D^{\pa{u, f}}$ (see \eqref{tensorD_u_f_bestVersion}) and which allows to interpret the corresponding integrability conditions as interpolations between those associated to conformally Einstein manifolds \eqref{GNintro1} and \eqref{GNintro2}, and those related to gradient Ricci solitons \eqref{piddu1} and \eqref{nonumber}. Moreover, $D^{\pa{u, f}}$ vanishes identically in the case of a conformally Einstein manifold, while it reduces to the tensor $D$ on a gradient Ricci soliton.
More precisely, in Section \ref{sec_ConfGradRS} we obtain two integrability conditions  for \eqref{EQ0.4} (see Theorems \ref{TH_FirstConditionCGRS} and \ref{CGRS_TH_SecCond}), which tell us that if $\varrg$ is a conformal gradient Ricci soliton then
\begin{equation}
  C_{ijk}-\sq{(m-2)u_t-f_t}W_{tijk} = D^{\pa{u, f}}_{ijk}
\end{equation}
and
\begin{equation}
    B_{ij}= \frac{1}{m-2}\set{D^{\pa{u, f}}_{ijk, k} -\pa{\frac{m-3}{m-2}}\sq{\pa{m-2}u_t-f_t}C_{jit}+\sq{f_tu_k+f_ku_t-(m-2)u_tu_k}W_{itjk}}.
  \end{equation}
In Section \ref{Sec_Conf_GenRS} we further extend our results to the very general case of a \emph{conformal generic Ricci soliton}, that is a Riemannian manifold $\varrg$ such that, for a conformal change of the metric $\tilde{g}=e^{2u}g$ with $u\in \cinf$, there exist a smooth vector field $X$, not necessarily a gradient, and a constant $\lambda$ such that
\[
\widetilde{\ricc} +\frac 12\mathcal{L}_X\tilde{g} = \lambda\tilde{g}.
\]
In this case the integrability conditions that we produce (see Theorems \ref{TH_FirstConditionCGenericRS} and \ref{TH_SecondConditionCGenericRS}) involve the construction of the appropriate generalization of both the tensors $D^X$ and  $D^{\pa{u, f}}$, that we call $D^{\pa{u, X}}$ and which reduces to the previous ones in the corresponding cases. As one can expect, these new conditions capture all those appearing in the aforementioned settings.

As it will become apparent to the reader, the analysis carried out in this paper is very heavy from the computational point of view; in order to ease the comprehension and also to provide help for future investigations, another aim of this paper is to present,  in a organized way, a number of useful formulas ranging from transformation laws for certain tensors to commutation rules for covariant derivatives that, to the best of our knowledge, are either difficult to find or not even present in the literature. In performing our calculations we exploit the \emph{moving frame} formalism, that turns out to be particularly appropriate for very long and involved computations like those appearing in our work.

The paper is organized as follows. In Section \ref{SecDef} we recall the relevant definitions and notation; in Section \ref{SecTrans} we compute the transformations laws of the previously introduced geometric objects under a conformal change of the underlying metric, while in Section \ref{SecCommy} we provide (and prove, in some particular cases) a number of useful commutation rules of covariant derivatives of functions, vector fields and geometric tensors. Sections \ref{sec5} and \ref{SecConfy} are brief reviews of results related to Ricci solitons and conformally Einstein manifolds, respectively. In Section \ref{sec_ConfGradRS} we study conformal gradient Ricci solitons, introducing the tensor $D^{\pa{u, f}}$ and the related integrability conditions. The subsequent Sections \ref{SecGenericRS} and \ref{Sec_Conf_GenRS} are devoted to the analysis of generic Ricci solitons and their conformal counterparts, involving the tensors $D^X$ and $D^\pa{u, X}$. In Section \ref{sec_higherorder} we come back to the case of gradient Ricci solitons and we  deduce the third and fourth integrability conditions. We end the paper with a final section in which we describe some interesting open problems, which - we hope - will inspire further investigations in these challenging but stimulating lines of research.

\section{Definitions and notation}\label{SecDef}

We begin by introducing some classical notions and objects we will be dealing with in the sequel  (see also \cite{MasRigSet} and \cite{CMMRGet}).

To perform computations, we use the moving frame notation with respect to a local orthonormal
 coframe. Thus we fix the index range $1\leq i, j, \ldots \leq m$ and recall that the Einstein summation convention will be in force throughout.

We denote by $\operatorname{R}$ the \emph{Riemann curvature tensor} (of type $\pa{1, 3}$) associated to the metric $g$, and by $\ricc$ and $S$ the corresponding \emph{Ricci tensor} and \emph{scalar curvature}, respectively.
The $(0, 4)$-versions of the Riemann curvature tensor and of the \emph{Weyl tensor} $\operatorname{W}$ are related in the following way:
\begin{equation}\label{Riemann_Weyl}
  R_{ijkt} = W_{ijkt} + \frac{1}{m-2}\pa{R_{ik}\delta_{jt}-R_{it}\delta_{jk}+R_{jt}\delta_{ik}-R_{jk}\delta_{it}}-\frac{S}{(m-1)(m-2)}\pa{\delta_{ik}\delta_{jt}-\delta_{it}\delta_{jk}}
\end{equation}
and they satisfy the  symmetry relations:
\begin{equation}
R_{ijkt} = -R_{jikt} = -R_{ijtk} = R_{ktij};
\end{equation}
\begin{equation}
W_{ijkt} = -W_{jikt} = -W_{ijtk} = W_{ktij}.
\end{equation}
A simple checking shows that the Weyl tensor is also totally trace-free and that it vanishes if $m=3$.

According to the above the (components of the) Ricci tensor and the scalar curvature are given by
\begin{equation}
  R_{ij} = R_{itjt} = R_{titj}
\end{equation}
and
\begin{equation}
  S = R_{tt}.
\end{equation}

The \emph{Schouten tensor} $\mathrm{A}$ is defined as
\begin{equation}\label{def_Schouten}
  \mathrm{A} = \ricc -\frac{S}{2(m-1)}g
\end{equation}

so that its trace  is
  \begin{equation}
    \operatorname{tr}(\mathrm{A}) = A_{tt} = \frac{(m-2)}{2(m-1)}S.
  \end{equation}


In terms of the Schouten tensor the decomposition of the Riemann curvature tensor reads as
\begin{equation}\label{decompRiemSchouten}
  \operatorname{R} = \textrm{W} + \frac{1}{m-2}\textrm{A}\owedge g,
\end{equation}
where $\owedge$ is the Kulkarni-Nomizu product; in components,
\begin{equation}\label{Riemann_Weyl_Schouten}
  R_{ijkt} = W_{ijkt} + \frac{1}{m-2}\pa{A_{ik}\delta_{jt}-A_{it}\delta_{jk}+A_{jt}\delta_{ik}-A_{jk}\delta_{it}}.
\end{equation}
We note that $\textrm{W}$ (more precisely, its $\pa{1, 3}$-version) is a conformal invariant (see e.g. \cite{MasRigSet}), hence  the above decomposition shows that the Schouten tensor is crucial in the study of conformal transformations.

The \emph{Cotton tensor} $C$ can be introduced as the obstruction for the Schouten tensor to be Codazzi, that is,
 \begin{equation}\label{def_Cotton_comp}
   C_{ijk} = A_{ij, k} - A_{ik, j} = R_{ij, k} - R_{ik, j} - \frac{1}{2(m-1)}\pa{S_k\delta_{ij}-S_j\delta_{ik}}.
 \end{equation}
 We recall that, for $m\geq 4$,  the Cotton tensor can also be defined as one of the possible divergences of the Weyl tensor:
 \begin{equation}\label{def_Cotton_comp_Weyl}
 C_{ijk}=\pa{\frac{m-2}{m-3}}W_{tikj, t}=-\pa{\frac{m-2}{m-3}}W_{tijk, t}.
 \end{equation}
 A computation shows that the two definitions coincide (see again \cite{MasRigSet}).
   The Cotton tensor enjoys  skew-symmetry in the second and third indices (i.e. $C_{ijk}=-C_{ikj}$) and furthermore is totally trace-free (i.e. $C_{iik}=C_{iki}=C_{kii}=0$).

 In what follows a relevant role will be played by the \emph{Bach tensor}, first introduced in general relativity by Bach, \cite{Bach}. Its componentwise definition is
 \begin{equation}\label{def_Bach_comp}
   B_{ij} = \frac{1}{m-3}W_{ikjl, lk} + \frac{1}{m-2}R_{kl}W_{ikjl} = \frac{1}{m-2}\pa{C_{jik, k}+R_{kl}W_{ikjl}}.
 \end{equation}

  A computation using the commutation rules for the second covariant derivative of the Weyl tensor or of the Schouten tensor (see the next section for both) shows that the Bach tensor is symmetric (i.e. $B_{ij}=B_{ji}$); it is also evidently trace-free (i.e. $B_{ii}=0$).  As a consequence we observe that we can write
  \[
   B_{ij}= \frac{1}{m-2}\pa{C_{ijk, k}+R_{kl}W_{ikjl}}.
  \]
It is worth reporting here the following interesting formula for the divergence of the Bach tensor (see e. g. \cite{CaoChen} for its proof)
\begin{equation}\label{diverBach}
  B_{ij, j} = \frac{m-4}{\pa{m-2}^2}R_{kt}C_{kti}.
\end{equation}

We also recall the definition of the \emph{ Einstein tensor}, which in components is given by
 \begin{equation}\label{def_Einstein_comp}
  E_{ij} = R_{ij} - \frac{S}{2}\delta_{ij}.
 \end{equation}


One of the main objects of our investigation are \emph{Ricci solitons}, which are defined through equation \eqref{EQ0.1_RicSolEq}; we explicitly note
that in components this latter  becomes
 \begin{equation}\label{RicSolEqComponents}
    R_{ij}+\frac{1}{2}\pa{X_{ij}+X_{ji}}=\lambda \delta_{ij}, \quad \lambda \in \erre
     \end{equation}
     and, in the gradient case,
 \begin{equation}\label{GradRicSolEqComponents}
   R_{ij}+f_{ij}=\lambda \delta_{ij}, \quad \lambda \in \erre.
 \end{equation}

The tensor $D$, introduced by Cao and Chen  in \cite{HDCaoChen_Steady}, turns out to be a fundamental tool in the study of the geometry of gradient Ricci solitons and, more in general, of gradient Einstein-type manifolds, as observed in \cite{CMMRGet}; in components it is defined as
 \begin{align}
   D_{ijk}=\frac{1}{m-2}\pa{f_kR_{ij}-f_jR_{ik}}+\frac{1}{(m-1)(m-2)}f_t\pa{R_{tk}\delta_{ij}-R_{tj}\delta_{ik}}-\frac{S}{(m-1)(m-2)}\pa{f_k\delta_{ij}-f_j \delta_{ik}}.
 \end{align}
 The $D$ tensor is skew-symmetric in the second and third indices (i.e. $D_{ijk}=-D_{ikj}$) and totally trace-free (i.e. $D_{iik}=D_{iki}=D_{kii}=0$).
Note that our convention for the tensor $D$ differs from that in \cite{CaoChen}.
A simple computation, using the definitions of the tensors involved, equation \eqref{GradRicSolEqComponents} and the fact that, for gradient Ricci solitons, the fundamental identity
 \begin{equation*}
   S_i = 2f_tR_{ti}
 \end{equation*}
 holds (see Section 4), shows that the tensor $D$ can be written in four equivalent ways:
    \begin{align}
   D_{ijk} &= \frac{1}{m-2}\pa{f_kR_{ij}-f_jR_{ik}}+\frac{1}{(m-1)(m-2)}f_t\pa{R_{tk}\delta_{ij}-R_{tj}\delta_{ik}}-\frac{S}{(m-1)(m-2)}\pa{f_k\delta_{ij}-f_j \delta_{ik}} \\  \nonumber &=\frac{1}{m-2}\pa{f_kR_{ij}-f_jR_{ik}}+\frac{1}{2(m-1)(m-2)}\pa{S_{k}\delta_{ij}-S_{j}\delta_{ik}}-\frac{S}{(m-1)(m-2)}\pa{f_k\delta_{ij}-f_j \delta_{ik}}\\ \nonumber &= \frac{1}{m-2}\pa{f_kA_{ij}-f_jA_{ik}} +\frac{1}{(m-1)(m-2)}f_t\pa{E_{tk}\delta_{ij}-E_{tj}\delta_{ik}}\\ \nonumber  &=\frac{1}{m-2}\pa{f_jf_{ik}-f_kf_{ij}}+\frac{1}{(m-1)(m-2)}f_t\pa{f_{tj}\delta_{ik}-f_{tk}\delta_{ij}}-\frac{\Delta f}{(m-1)(m-2)}\pa{f_j\delta_{ik}-f_k \delta_{ij}} .
 \end{align}

\section{Transformation laws  under a conformal change of the metric}\label{SecTrans}

Let $\riemanng{M}$ be a Riemannian manifold of dimension $m\geq 3$. The moving frame formalism is extremely useful in the calculation of the transformation laws of geometric tensors under a conformal change of the metric and in the derivation of commutation rules, as we shall see in the next section.
For the sake of completeness (see \cite{MasRigSet} for details)  we recall that, having fixed a (local) orthonormal coframe $\set{\theta^i}$, $i=1, \ldots, m$ with dual frame $\set{e_i}$, $i=1, \ldots, m$,  the corresponding \emph{Levi-Civita connection forms} $\set{\theta^i_j}$, $i, j =1, \ldots, m$ are the unique $1$-forms satisfying
\begin{align}
  &d\theta^i = -\theta^i_j \wedge
  \theta^j \quad \text{(first structure equations)}, \label{1_firstStructureEq} \\
  &\theta^i_j + \theta^j_i = 0. \label{1_skewsymmConnForm}
\end{align}
The \emph{curvature forms} $\set{\Theta^i_j}$, $i, j =1, \ldots, m$, associated to the coframe
are the $2$-forms defined through the  \emph{second structure equations}
\begin{equation}\label{1_secondStructureEq}
  {d\theta^i_j = -\theta^i_k \wedge \theta^k_j + \Theta^i_j.}
\end{equation}
They are skew-symmetric (i.e. $\Theta^i_j + \Theta^j_i = 0$) and they can be written as
\begin{equation}\label{1_def_forme_curvatura}
  \Theta^i_j = \frac{1}{2}R^i_{jkt}\theta^k \wedge \theta^t = \sum_{k<t}R^i_{jkt}\theta^k \wedge \theta^t,
\end{equation}
where $R^i_{jkt}$ are precisely the coefficients of the ($(1, 3)$-version of the) Riemann curvature tensor.

 The \emph{covariant derivative of a vector field} $X \in \mathfrak{X}(M)$ is defined as
\[
\nabla X = (dX^i + X^j\theta^i_j)\otimes e_i=X^i_k\theta^k \otimes e_i,
\]
while the \emph{covariant derivative of a $1$-form} $\omega$ is defined as
\[
\nabla \omega = (d\omega_i-w_j\theta^j_i)\otimes \theta^i= \omega_{ik}\theta^k \otimes \theta^i.
\]
The \emph{divergence} of the vector field $X\in\mathfrak{X}(M)$  is the trace of $\nabla X$, that is,
\begin{equation}\label{1_divergence}
  \operatorname{div}X = \operatorname{tr}\pa{\nabla X} = g\pa{\nabla_{e_i}X, e_i} = X_i^i.
\end{equation}
For a function $f\in \cinf$ we can write
\begin{equation}\label{DifferentialComponents}
df = f_i\theta^i,
\end{equation}
for some smooth coefficients $f_i\in\cinf$. The \emph{Hessian} of $f$, $\hess(f)$, is the $(0, 2)$-tensor defined as
\begin{equation}
  \hess(f) = \nabla df = f_{ij}\theta^j\otimes\theta^i,
\end{equation}
with
\begin{equation}\label{HessianComponents}
f_{ij}\theta^j = df_i - f_t\theta_i^t
\end{equation}
and
\[
f_{ij}=f_{ji}
\]
(see also next section).
The \emph{Laplacian} of $f$ is the trace of the Hessian, that is
\[
\Delta f = \operatorname{tr}(\hess(f)) = f_{ii}.
\]
Now we are ready to recall (and prove, in some cases) the transformation laws that will be useful in our computations.

We consider the conformal change of the metric (written in ``exponential form'')
\begin{equation}\label{eq_conformal_change_of_g}
\tilde{g}=e^{2u} g, \qquad u\in\cinf;
\end{equation}
$e^u$ is called the \emph{stretching factor} of the conformal change.
We use the superscript $\,\,\tilde{}\,\,$ to denote quantities related to the metric $\tilde{g}$.

It is obvious that  in the new metric $\tilde{g}$ the $1$-forms
\begin{equation}\label{2.2}
\widetilde{\theta}^{i}=e^u\theta^{i},\qquad i=1,...,m,
\end{equation}
give a local orthonormal coframe. It is easy to deduce that, if $du=u_{t}\theta^{t}$, the $1$-forms
\begin{equation}\label{ConfTransfConnectionForms}
\tilde{\theta}_{j}^{i}=\theta_{j}^{i}+u_j\theta
^{i}-u_i\theta^{j}
\end{equation}
are skew-symmetric and satisfy the first structure equation. Hence, they are
the connection forms relative to the coframe defined in \eqref{2.2}. A straightforward computation using the structure equations and \eqref{HessianComponents} shows that the curvature forms relative to the coframe \eqref{2.2} are
\begin{equation}\label{ConfChangeCurvatureForms}
  \tilde{\Theta}^i_j = \Theta^i_j +\sq{\pa{u_{jk}-u_ju_k}\delta^i_{t}-\pa{u_{ik}-u_iu_k}\delta^j_{t}-\abs{\nabla u}^2\delta^i_k\delta^j_t}\theta^k\wedge\theta^t.
\end{equation}
Equation \eqref{ConfChangeCurvatureForms} is the starting point for the next transformation laws that we list without further comments.
\begin{itemize}
    \item \textbf{Riemann curvature tensor}:
    \begin{equation}\label{Riemannexp}
      e^{2u}\tilde{R}_{ijkt}=R_{ijkt}+\pa{u_{jk}-u_ju_k}\delta_{it}-\pa{u_{jt}-u_ju_t}\delta_{ik}-\pa{u_{ik}-u_iu_k}\delta_{jt}+\pa{u_{it}-u_iu_t}\delta_{jk}-\abs{\nabla u}^2\pa{\delta_{ik}\delta_{jt}-\delta_{it}\delta_{jk}}.
    \end{equation}
    \begin{proof}
      The previous equation follows easily skew-symmetrizing the coefficients of the wedge products on the right hand side of \eqref{ConfChangeCurvatureForms}, and
recalling  equation \eqref{1_def_forme_curvatura}.
    \end{proof}
Tracing \eqref{Riemannexp} we get
  \item \textbf{Ricci tensor}:
  \begin{equation}\label{Ricciexp}
  \widetilde{\ricc} = \ricc -\pa{m-2}\hess\pa{u}+\pa{m-2}du\otimes du-\Delta u\, g-\pa{m-2}\abs{\nabla u}^2 g,
  \end{equation}
  that, in components, reads as
  \begin{equation}\label{RicciexpComponents}
  e^{2u}\widetilde{R}_{ij} = R_{ij} -\pa{m-2}u_{ij}+\pa{m-2}u_iu_j-\Delta u\,\delta_{ij}-\pa{m-2}\abs{\nabla u}^2\delta_{ij}.
  \end{equation}
  Tracing \eqref{Ricciexp} we deduce
  \item  \textbf{Scalar curvature}:
   \begin{equation}\label{scalarExp}
  e^{2 u}\tilde{S} = S -2\pa{m-1}\Delta u-\pa{m-1}\pa{m-2}\abs{\nabla u}^2.
  \end{equation}
  Next we derive the transformation laws for the
   \item \textbf{Covariant derivative of the Ricci tensor}:
  \begin{align}\label{NablaRicciexpComponents}
  e^{3u}\widetilde{R}_{ij, k} &= R_{ij, k} -\pa{m-2}u_{ijk}-\pa{u_{ttk}-2u_k \Delta u}\delta_{ij}\\ \nonumber &-\pa{2R_{ij}u_k+u_iR_{jk}+u_jR_{ik}}+u_t\pa{R_{ti}\delta_{jk}+R_{tj}\delta_{ik}} \\ \nonumber &+2(m-2)\pa{u_iu_{jk}+u_ju_{ik}+u_ku_{ij}}-(m-2)u_t\pa{u_{ti}\delta_{jk}+u_{tj}\delta_{ik}+2u_{tk}\delta_{ij}} \\ \nonumber &-4(m-2)u_iu_ju_k +(m-2)\abs{\nabla u}^2\pa{u_i\delta_{jk}+u_j\delta_{ik}+2u_k\delta_{ij}}.
  \end{align}
  \begin{proof}
    The definition of covariant derivative implies that
    \begin{equation}\label{CovDerivRicci}
      R_{ij, k}\theta^k = dR_{ij}-R_{tj}\theta^t_i-R_{it}\theta^t_j.
    \end{equation}
    Now equation \eqref{NablaRicciexpComponents} follows from \eqref{CovDerivRicci}, from the fact that
    \[
    e^{3u}\tilde{R}_{ij, k}\theta^k = d\pa{e^{2u}\tilde{R}_{ij}}-\pa{e^{2u}\tilde{R}_{tj}}\tilde{\theta^t_i}-\pa{e^{2u}\tilde{R}_{it}}\tilde{\theta^t_j}-\tilde{R}_{ij}d\pa{e^{2u}}
    \]
    and from \eqref{RicciexpComponents}.
  \end{proof}

   \item \textbf{Second Covariant derivative of the Ricci tensor}:
\begin{align}\label{ExpochangenablasquaredRicci}
  e^{4 u}\tilde{R}_{ij, kt} &= R_{ij, kt} -  (m-2)u_{ijkt}-u_{sskt}\delta_{ij}+3\pa{u_tu_{ssk}+u_ku_{sst}}\delta_{ij}-g\pa{\nabla u, \nabla\Delta u}\delta_{ij}\delta_{kt}\\\nonumber &+2\Delta u\pa{u_{kt}-4u_ku_t+\abs{\nabla u}^2\delta_{kt}}\delta_{ij}\\ \nonumber &+u_lR_{li, t}\delta_{jk}+u_lR_{lj, t}\delta_{ik}+u_lR_{il, k}\delta_{jt}+u_lR_{lj, k}\delta_{it}+u_lR_{ij, l}\delta_{kt}+ R_{il}u_{lt}\delta_{jk}+R_{jl}u_{lt}\delta_{ik}\\ \nonumber &-\pa{u_iR_{jk, t}+u_jR_{ik, t}+u_iR_{jt, k}+u_jR_{it, k}+3u_kR_{ij, t}+3u_tR_{ij, k}}\\ \nonumber &-\pa{u_{it}R_{jk}+u_{jt}R_{ik}+2u_{kt}R_{ij}}\\ \nonumber &+(m-2)\pa{2u_iu_{jkt}+u_iu_{jtk}+2u_ju_{ikt}+u_ju_{itk}+3u_ku_{ijt}+3u_tu_{ijk}}\\ \nonumber &+2(m-2)\pa{u_{ij}u_{kt}+u_{ik}u_{jt}+u_{jk}u_{it}}-(m-2)\pa{2u_lu_{lkt}\delta_{ij}+u_lu_{ljt}\delta_{ik}+u_lu_{lit}\delta_{jk}}\\ \nonumber &-(m-2)\pa{2u_{kl}u_{lt}\delta_{ij}+u_{jl}u_{lt}\delta_{ik}+u_{il}u_{lt}\delta_{jk}}\\ \nonumber
  &-\pa{R_{tl}u_lu_i\delta_{jk}+R_{tl}u_lu_j\delta_{ik}+3R_{il}u_lu_t\delta_{jk}+3R_{jl}u_lu_t\delta_{ik}}+\operatorname{Ric}\pa{\nabla u, \nabla u}\pa{\delta_{jk}\delta_{it}+\delta_{ik}\delta_{jt}}\\ \nonumber &+4\pa{u_iu_tR_{jk}+u_ju_tR_{ik}+2u_ku_tR_{ij}}+\pa{2u_iu_jR_{kt}+3u_iu_kR_{jt}+3u_ju_kR_{it}}\\
  \nonumber &-8(m-2)\pa{u_iu_ju_{tk}+u_iu_ku_{jt}+u_ju_ku_{it}+u_iu_tu_{jk}+u_ju_tu_{ik}+u_ku_tu_{ij}} \\\nonumber &-(m-2)\pa{u_lu_{ljk}\delta_{it}+u_lu_{lik}\delta_{jt}+u_lu_{ijl}\delta_{kt}}-\abs{\nabla u}^2\pa{R_{jk}\delta_{it}+R_{ik}\delta_{jt}+2R_{ij}\delta_{kt}} \\ \nonumber &-\pa{u_ju_lR_{lk}\delta_{it}+u_iu_lR_{lk}\delta_{jt}+u_iu_lR_{lj}\delta_{kt}+u_ju_lR_{li}\delta_{kt}+2u_ku_lR_{lj}\delta_{it}+2u_ku_lR_{li}\delta_{jt}}\\ \nonumber &+3(m-2)\pa{u_iu_lu_{lt}\delta_{jk}+u_ju_lu_{lt}\delta_{ik}+2u_ku_lu_{lt}\delta_{ij}+u_iu_lu_{lt}\delta_{jk}+u_ju_lu_{lt}\delta_{ik}+2u_ku_lu_{lt}\delta_{ij}}\\ \nonumber
  &+2(m-2)\pa{u_iu_lu_{lk}\delta_{jt}+u_ju_lu_{lk}\delta_{it}+u_iu_lu_{lj}\delta_{kt}+u_ju_lu_{li}\delta_{kt}+u_ku_lu_{li}\delta_{jt}+u_ku_lu_{lj}\delta_{it}}\\ \nonumber &+(m-2)\abs{\nabla u}^2\pa{u_{it}\delta_{jk}+u_{jt}\delta_{ik}+2u_{kt}\delta_{ij}+2u_{ij}\delta_{kt}+2u_{ik}\delta_{jt}+2u_{jk}\delta_{it}}\\ \nonumber &-(m-2)\hess\pa{u}\pa{\nabla u, \nabla u}\pa{\delta_{jk}\delta_{it}+\delta_{ik}\delta_{jt}+2\delta_{ij}\delta_{kt}}\\ \nonumber &+24(m-2)u_iu_ju_ku_t\\ \nonumber &-4(m-2)\abs{\nabla u}^2\pa{u_ju_k\delta_{it}+u_iu_k\delta_{jt}+u_iu_j\delta_{kt}+u_iu_t\delta_{jk}+u_ju_t\delta_{ik}+2u_ku_t\delta_{ij}}\\ \nonumber &+(m-2)\abs{\nabla u}^4\pa{\delta_{jk}\delta_{it}+\delta_{ik}\delta_{jt}+2\delta_{ij}\delta_{kt}}.
  \end{align}
    The proof of \eqref{ExpochangenablasquaredRicci} is just  a really long computation, similar to the one performed to obtain equation \eqref{NablaRicciexpComponents}.

   \item  \textbf{Differential of the scalar curvature}:
   \begin{equation}\label{NablascalarExp}
  e^{3 u}\tilde{S}_k = S_k -2\pa{m-1}u_{ttk}-2\pa{m-1}\pa{m-2}u_tu_{tk}-2\sq{S-2(m-1)\Delta u-(m-1)(m-2)\abs{\nabla u}^2}u_k.
  \end{equation}
  \begin{proof}
    It follows from the fact that $e^{3u}\tilde{S}_k\theta^k = e^{2u}d\tilde{S} = d\pa{e^{2u}\tilde{S}}-\tilde{S}d\pa{e^{2u}}$ and from \eqref{scalarExp}.
  \end{proof}

  \item  \textbf{Hessian of the scalar curvature}:
   \begin{align}\label{HessianscalarExp}
  e^{4 u}\tilde{S}_{kt} &= S_{kt} -2\pa{m-1}u_{sskt}-2\pa{m-1}\pa{m-2}u_{ks}u_{st}-2\pa{m-1}\pa{m-2}u_su_{skt}+6(m-1)\pa{u_ku_{sst}+u_tu_{ssk}}\\ \nonumber &+6(m-1)(m-2)\pa{u_su_{sk}u_t+u_su_{st}u_k} -3\pa{S_tu_k+S_ku_t}\\ \nonumber &-2\sq{S-2(m-1)\Delta u-(m-1)(m-2)\abs{\nabla u}^2}\pa{u_{kt}-4u_ku_t+\abs{\nabla u}^2\delta_{kt}}  \\ \nonumber &+\sq{g\pa{\nabla S, \nabla u}-2(m-1)g\pa{\nabla u, \nabla \Delta u}-2(m-1)(m-2)\operatorname{Hess}(u)\pa{\nabla u, \nabla u}}\delta_{kt}.
  \end{align}
  \begin{proof}
    Equation \eqref{HessianscalarExp} follows from the fact that $e^{4u}\tilde{S}_{kt}\theta^t = d\pa{e^{3u}\tilde{S}_k}-\tilde{S}_kd\pa{e^{3u}}-e^{3u}\tilde{S}_t\tilde{\theta}^t_k$ and from \eqref{scalarExp} and \eqref{NablascalarExp}. Alternatively, \eqref{HessianscalarExp} can be obtained tracing \eqref{ExpochangenablasquaredRicci} with respect to $i$ and $j$.
  \end{proof}
  Tracing \eqref{HessianscalarExp} and using \eqref{TracedThirdDerivFunctionRicci} (see next Section) we deduce
    \item  \textbf{Laplacian of the scalar curvature}:
    \begin{align}\label{LaplacianscalarExp}
      e^{4u}\tilde{\Delta}\tilde{S} &= \Delta S -2(m-1)\Delta^2u-2(m-1)(m-2)\abs{\hess(u)}^2\\ \nonumber &-2(m-1)(m-2)\ricc\pa{\nabla u, \nabla u}-4(m-1)(m-4)g\pa{\nabla u, \nabla\Delta u)}\\ \nonumber &-2(m-1)(m-2)(m-6)\hess(u)\pa{\nabla u, \nabla u}+(m-6)g\pa{\nabla S, \nabla u}-2S\Delta u +4(m-1)\pa{\Delta u}^2 \\ \nonumber &+2(m-1)(3m-10)\abs{\nabla u}^2\Delta u+2(m-1)(m-2)(m-4)\abs{\nabla u}^4-2(m-4)S\abs{\nabla u}^2.
    \end{align}

  \item \textbf{the Hessian of a function $f \in \cinf$}:
  \begin{equation}
    \widetilde{\operatorname{Hess}}(f)=\hess(f)-\pa{df\otimes du+ du \otimes df}+g\pa{\nabla f, \nabla u}g,
  \end{equation}
  which in components reads as
  \begin{equation}\label{HessianExpComp}
    e^{2u}\tilde{f}_{ij}=f_{ij}-\pa{f_iu_j+f_ju_i}+(f_tu_t)\delta_{ij}.
  \end{equation}
  \begin{proof}
    From $du=u_i\theta^i=\tilde{u}_i\tilde{\theta}^i$ we deduce that
    \begin{equation}\label{diffComponentsExpTransf}
    \tilde{u}_i = e^{-u}u_i.
    \end{equation}
    Now \eqref{HessianExpComp} follows from a straightforward computation using \eqref{diffComponentsExpTransf}, \eqref{HessianComponents} and \eqref{ConfTransfConnectionForms}.
  \end{proof}
  Tracing \eqref{HessianExpComp} we get
  \item \textbf{the Laplacian of a function $f \in \cinf$}:
   \begin{equation}\label{LaplacianExpComp}
    e^{2u}\tilde{\Delta}f=\Delta f+(m-2)g\pa{\nabla f, \nabla u} = f_{tt}+(m-2)f_tu_t.
  \end{equation}
\item \textbf{the third derivative of a function $f \in \cinf$}:
\begin{align}\label{thirdDerivFunctExpComp}
e^{3u}\tilde{f}_{ijk} &= f_{ijk} -2\pa{f_{ij}u_k+f_{ik}u_j+f_{jk}u_i}-\pa{f_iu_{jk}+f_ju_{ik}} +3\pa{f_iu_j+f_ju_i}u_k +2u_iu_jf_k \\ \nonumber &+u_t\pa{f_{tk}\delta_{ij}+f_{tj}\delta_{ik}+f_{ti}\delta_{jk}}+f_tu_{tk}\delta_{ij}-(f_tu_t)\pa{u_i\delta_{jk}+u_j\delta_{ik}+2u_k\delta_{ij}}-\abs{\nabla u}^2\pa{f_i\delta_{jk}+f_j\delta_{ik}};
\end{align}
in particular,
\begin{equation}\label{thirdDerivFunctExpCompTraced}
e^{3u}\tilde{f}_{ttk}=f_{ttk}-2\Delta f u_k +(m-2)\sq{f_tu_{tk}+u_tf_{tk}-2(f_tu_t)u_k}.
\end{equation}

\begin{proof}
By definition of covariant derivative we have
\[
\tilde{f}_{ijk}\tilde{\theta}^k = d \tilde{f}_{ij}-\tilde{f}_{tj}\tilde{\theta}^t_i-\tilde{f}_{it}\tilde{\theta}^t_j,
\]
which can be written as
\[
e^{3u}\tilde{f}_{ijk}\theta^k = e^{2u}d \tilde{f}_{ij}-e^{2u}\tilde{f}_{tj}\tilde{\theta}^t_i-e^{2u}\tilde{f}_{it}\tilde{\theta}^t_j= d\pa{e^{2u}\tilde{f}_{ij}}-\tilde{f}_{ij}d(e^{2u})-e^{2u}\tilde{f}_{tj}\tilde{\theta}^t_i-e^{2u}\tilde{f}_{it}\tilde{\theta}^t_j.
\]
Now equation \eqref{thirdDerivFunctExpComp} follows using \eqref{HessianExpComp}, \eqref{ConfTransfConnectionForms} and simplifying.
\end{proof}

  \item \textbf{Schouten tensor}:
  \begin{equation}\label{Schoutenexp}
   \tilde{\mathrm{A}} = \mathrm{A} - (m-2)\hess\pa{u} +  (m-2)du\otimes du - \pa{\frac{m-2}{2}}\abs{\nabla u}^2 g,
  \end{equation}
  which in components reads as
  \begin{equation}\label{SchoutenexpComponents}
e^{2 u}\tilde{A}_{ij} = A_{ij} - (m-2)u_{ij}+(m-2)u_iu_j-\pa{\frac{m-2}{2}}\abs{\nabla u}^2\delta_{ij}.
  \end{equation}
  The proof of \eqref{Schoutenexp} follows easily from the definition of the Schouten tensor and from \eqref{Ricciexp} and \eqref{scalarExp}.

 \item \textbf{Covariant derivative of the Schouten tensor}:
 \begin{align}\label{ExpochangenablaSchouten}
  e^{3u}\tilde{A}_{ij, k} &= A_{ij, k} -  (m-2)u_{ijk} + u_lA_{li}\delta_{jk}+u_lA_{lj}\delta_{ik}-\pa{u_iA_{jk}+u_jA_{ik}+2u_kA_{ij}}\\ \nonumber &+2(m-2)\pa{u_iu_{jk}+u_ju_{ik}+u_ku_{ij}} - (m-2)\pa{u_lu_{lk}\delta_{ij}+u_lu_{lj}\delta_{ik}+u_lu_{li}\delta_{jk}}\\ \nonumber &-4 (m-2)u_iu_ju_k +(m-2)\abs{\nabla u}^2\pa{u_i\delta_{jk}+u_j\delta_{ik}+u_k\delta_{ij}}.
\end{align}
\begin{proof}
    The definition of covariant derivative implies that
    \begin{equation}\label{CovDerivSchouten}
      A_{ij, k}\theta^k = dA_{ij}-A_{tj}\theta^t_i-A_{it}\theta^t_j.
    \end{equation}
    Now equation \eqref{ExpochangenablaSchouten} follows from \eqref{CovDerivSchouten}, from the fact that
    \[
    e^{3u}\tilde{A}_{ij, k}\theta^k = d\pa{e^{2u}\tilde{A}_{ij}}-\pa{e^{2u}\tilde{A}_{tj}}\tilde{\theta^t_i}-\pa{e^{2u}\tilde{A}_{it}}\tilde{\theta^t_j}-\tilde{A}_{ij}d\pa{e^{2u}}
    \]
    and from \eqref{SchoutenexpComponents}.
  \end{proof}

 \item \textbf{Second Covariant derivative of the Schouten tensor}:
\begin{align}\label{ExpochangenablasquaredSchouten}
  e^{4 u}\tilde{A}_{ij, kt} &= A_{ij, kt} -  (m-2)u_{ijkt}\\ \nonumber &+u_lA_{li, t}\delta_{jk}+u_lA_{lj, t}\delta_{ik}+u_lA_{il, k}\delta_{jt}+u_lA_{lj, k}\delta_{it}+u_lA_{ij, l}\delta_{kt}+ A_{il}u_{lt}\delta_{jk}+A_{jl}u_{lt}\delta_{ik}\\ \nonumber &-\pa{u_iA_{jk, t}+u_jA_{ik, t}+u_iA_{jt, k}+u_jA_{it, k}+3u_kA_{ij, t}+3u_tA_{ij, k}}\\ \nonumber &-\pa{u_{it}A_{jk}+u_{jt}A_{ik}+2u_{kt}A_{ij}}\\ \nonumber &+(m-2)\pa{2u_iu_{jkt}+u_iu_{jtk}+2u_ju_{ikt}+u_ju_{itk}+3u_ku_{ijt}+3u_tu_{ijk}}\\ \nonumber &+2(m-2)\pa{u_{ij}u_{kt}+u_{ik}u_{jt}+u_{jk}u_{it}}-(m-2)\pa{u_lu_{lkt}\delta_{ij}+u_lu_{ljt}\delta_{ik}+u_lu_{lit}\delta_{jk}}\\ \nonumber &-(m-2)\pa{u_{kl}u_{lt}\delta_{ij}+u_{jl}u_{lt}\delta_{ik}+u_{il}u_{lt}\delta_{jk}}\\ \nonumber
  &-\pa{A_{tl}u_lu_i\delta_{jk}+A_{tl}u_lu_j\delta_{ik}+3A_{il}u_lu_t\delta_{jk}+3A_{jl}u_lu_t\delta_{ik}}+\mathrm{A}\pa{\nabla u, \nabla u}\pa{\delta_{jk}\delta_{it}+\delta_{ik}\delta_{jt}}\\ \nonumber &+4\pa{u_iu_tA_{jk}+u_ju_tA_{ik}+2u_ku_tA_{ij}}+\pa{2u_iu_jA_{kt}+3u_iu_kA_{jt}+3u_ju_kA_{it}}\\
   \nonumber &-8(m-2)\pa{u_iu_ju_{tk}+u_iu_ku_{jt}+u_ju_ku_{it}+u_iu_tu_{jk}+u_ju_tu_{ik}+u_ku_tu_{ij}} \\
  \nonumber &-(m-2)\pa{u_lu_{ljk}\delta_{it}+u_lu_{lik}\delta_{jt}+u_lu_{ijl}\delta_{kt}}-\abs{\nabla u}^2\pa{A_{jk}\delta_{it}+A_{ik}\delta_{jt}+2A_{ij}\delta_{kt}} \\ \nonumber &-\pa{u_ju_lA_{lk}\delta_{it}+u_iu_lA_{lk}\delta_{jt}+u_iu_lA_{lj}\delta_{kt}+u_ju_lA_{li}\delta_{kt}+2u_ku_lA_{lj}\delta_{it}+2u_ku_lA_{li}\delta_{jt}}\\ \nonumber &+(m-2)\left(3u_iu_lu_{lt}\delta_{jk}+3u_ju_lu_{lt}\delta_{ik}+3u_ku_lu_{lt}\delta_{ij}+2u_iu_lu_{lk}\delta_{jt}+2u_iu_lu_{lj}\delta_{kt}+2u_ju_lu_{lk}\delta_{it}\right.\\ \nonumber &\quad+\left.2u_ku_lu_{lj}\delta_{it}+2u_ju_lu_{li}\delta_{kt}+2u_ku_lu_{li}\delta_{jt}+3u_lu_tu_{lk}\delta_{ij}+3u_lu_tu_{lj}\delta_{ik}+3u_lu_tu_{li}\delta_{jk}\right)\\ \nonumber &+\abs{\nabla u}^2(m-2)\pa{u_{it}\delta_{jk}+u_{jt}\delta_{ik}+u_{kt}\delta_{ij}+2u_{ij}\delta_{kt}+2u_{ik}\delta_{jt}+2u_{jk}\delta_{it}}\\ \nonumber &-(m-2)\hess\pa{u}\pa{\nabla u, \nabla u}\pa{\delta_{jk}\delta_{it}+\delta_{ik}\delta_{jt}+\delta_{ij}\delta_{kt}}\\ \nonumber &+24(m-2)u_iu_ju_ku_t\\ \nonumber &-4(m-2)\abs{\nabla u}^2\pa{u_ju_k\delta_{it}+u_iu_k\delta_{jt}+u_iu_j\delta_{kt}+u_iu_t\delta_{jk}+u_ju_t\delta_{ik}+u_ku_t\delta_{ij}}\\ \nonumber &+(m-2)\abs{\nabla u}^4\pa{\delta_{jk}\delta_{it}+\delta_{ik}\delta_{jt}+\delta_{ij}\delta_{kt}}.
  \end{align}
    The proof of \eqref{ExpochangenablasquaredSchouten} is just  a really long computation, similar to the one performed to obtain equation \eqref{CovDerivSchouten}.
\begin{rem}
  Equations \eqref{ExpochangenablaSchouten} and \eqref{ExpochangenablasquaredSchouten} can be also obtained from the corrisponding relations for the Ricci tensor, with the aid of \eqref{scalarExp}, \eqref{NablascalarExp} and \eqref{HessianscalarExp}.
\end{rem}

  \item \textbf{Weyl tensor} ($\pa{1, 3}$-version):
  \begin{equation}\label{Weylexp}
  e^{2u}\tilde{W}^i_{jkt} = W^i_{jkt}
    \end{equation}
    For the proof of \eqref{Weylexp} we refer to \cite{MasRigSet}, Chapter 2.
    \item \textbf{Cotton tensor}:
  \begin{equation}\label{Cottonlexp}
  e^{3u}\tilde{C}_{ijk} = C_{ijk}-(m-2)u_tW_{tijk}.
    \end{equation}
    \begin{proof}
     From the definition of the Cotton tensor and from \eqref{CovDerivSchouten} we have
      \[
      e^{3u}\tilde{C}_{ijk} = e^{3u}\tilde{A}_{ij, k}-e^{3u}\tilde{A}_{ik, j} = A_{ij, k}- A_{ik, j}-(m-2)\pa{u_{ijk}-u_{ikj}}+u_t\pa{A_{tj}\delta_{ik}-A_{tk}\delta_{ij}}+u_jA_{ik}-u_kA_{ij}.
      \]
      Equation \eqref{Cottonlexp} now follows using \eqref{commutatioThirdDerFunctWeilSchouten} (see next section) and simplifying.
    \end{proof}

 \item \textbf{Bach tensor}:
   \begin{equation}\label{BachExpComp}
    e^{4u}\tilde{B}_{ij}=B_{ij} + (m-4)\sq{u_tu_kW_{tikj}+\frac{1}{m-2}\pa{C_{ijt}+C_{jit}}u_t}.
  \end{equation}
  \begin{proof}(sketch)
    From the definition of the Bach tensor we have
    \begin{equation*}
      e^{4u}\tilde{B}_{ij} = \frac{e^{4u}}{m-2}\sq{\tilde{C}_{ijt, t}+\tilde{R}_{kl}\tilde{W}_{ikjl}}=\frac{1}{m-2}\sq{e^{4u}\pa{\tilde{A}_{ij, tt}-\tilde{A}_{it, jt}}+\pa{e^{4u}\tilde{R}_{kl}\tilde{W}_{ikjl}}}.
    \end{equation*}
    The second term on the right hand side is easily computed using \eqref{RicciexpComponents} and \eqref{Weylexp}:
    \[
    e^{4u}\tilde{R}_{kl}\tilde{W}_{ikjl}=R_{kl}W_{ikjl}-(m-2)u_{kl}W_{ikjl}+(m-2)u_ku_lW_{ikjl}.
    \]
    As far as the first term is concerned, we trace \eqref{ExpochangenablasquaredSchouten} with respect to the third and fourth indices and then with respect to the second and the fourth, then we simplify with a lot of patience. Summing up we finally obtain \eqref{BachExpComp}.
  \end{proof}

Using the previous relations (in particular equations \eqref{HessianExpComp}, \eqref{RicciexpComponents}, \eqref{NablaRicciexpComponents}, \eqref{NablascalarExp}) and the fact that $e^u\tilde{f}_t=f_t$, we can prove, with a really long but straightforward calculation, the tranformation laws for $D$ and $\nabla D$.
\item \textbf{$D$ tensor}: If $\pa{M, \tilde{g}, f, \lambda}$ is a soliton structure, then
   \begin{align}\label{DExpComp}
    e^{3u}\tilde{D}_{ijk}&= \frac{1}{m-2}\pa{f_kR_{ij}-f_jR_{ik}}+\frac{1}{(m-1)(m-2)}f_t\pa{R_{tk}\delta_{ij}-R_{tj}\delta_{ik}}-\frac{S}{(m-1)(m-2)}\pa{f_k\delta_{ij}-f_j \delta_{ik}}\\ \nonumber &+ u_i\pa{f_ku_j-f_ju_k} + f_ju_{ik}-f_ku_{ij} \\ \nonumber &+\frac{1}{m-1}\sq{\Delta u\pa{f_k\delta_{ij}-f_j\delta_{ik}}+f_t\pa{u_{tj}\delta_{ik}-u_{tk}\delta_{ij}}+ \pa{f_tu_t}\pa{u_k\delta_{ij}-u_j\delta_{ik}}-\abs{\nabla u}^2\pa{f_k\delta_{ij}-f_j\delta_{ik}}}.
\end{align}
Viceversa, if $\pa{M, g, f, \lambda}$ is a soliton structure, then we have
   \begin{align}\label{DExpCompStartingFrom}
    e^{3u}&\set{\frac{1}{m-2}\pa{\tilde{f}_k\tilde{R}_{ij}-\tilde{f}_j\tilde{R}_{ik}}+\frac{1}{(m-1)(m-2)}\tilde{f}_t\pa{\tilde{R}_{tk}\delta_{ij}-\tilde{R}_{tj}\delta_{ik}}-\frac{\tilde{S}}{(m-1)(m-2)}\pa{\tilde{f}_k\delta_{ij}-\tilde{f}_j \delta_{ik}}}\\ \nonumber &= D_{ijk}+ u_i\pa{f_ku_j-f_ju_k} + f_ju_{ik}-f_ku_{ij} \\ \nonumber &+\frac{1}{m-1}\sq{\Delta u\pa{f_k\delta_{ij}-f_j\delta_{ik}}+f_t\pa{u_{tj}\delta_{ik}-u_{tk}\delta_{ij}}+ \pa{f_tu_t}\pa{u_k\delta_{ij}-u_j\delta_{ik}}-\abs{\nabla u}^2\pa{f_k\delta_{ij}-f_j\delta_{ik}}}.
\end{align}

\item \textbf{Covariant derivative of the $D$ tensor}: If $\pa{M, \tilde{g}, f, \lambda}$ is a soliton structure, then
 \begin{align}\label{CovDerivDExpComp}
    e^{4u}\tilde{D}_{ijk, t}&= \frac{1}{m-2}\sq{\pa{f_{kt}R_{ij}-f_{jt}R_{ik}}+\pa{f_kR_{ij, t}-f_jR_{ik, t}}}\\ \nonumber &+\frac{1}{(m-1)(m-2)}\sq{f_{st}\pa{R_{sk}\delta_{ij}-R_{sj}\delta_{ik}}+f_s\pa{R_{sk, t}\delta_{ij}-R_{sj, t}\delta_{ik}}}\\ \nonumber &-\frac{1}{(m-1)(m-2)}\sq{S_t\pa{f_k\delta_{ij}-f_j\delta_{ik}}+S\pa{f_{kt}\delta_{ij}-f_{jt}\delta_{ik}}} \\ \nonumber &+\pa{u_{ik}f_{jt}-u_{ij}f_{kt}} + \pa{u_{ikt}f_j-u_{ijt}f_k} +\pa{u_iu_jf_{kt}-u_iu_kf_{jt}}+\frac{1}{m-1}u_{sst}\pa{f_k\delta_{ij}-f_j\delta_{ik}}\\ \nonumber &-\frac{3}{m-2}u_t\pa{f_kR_{ij}-f_jR_{ik}}-\frac{1}{m-1}f_s\pa{u_{skt}\delta_{ij}-u_{sjt}\delta_{ik}}-\frac{1}{m-2}f_t\pa{u_kR_{ij}-u_jR_{ik}}\\ \nonumber &+\frac{1}{m-2}\pa{f_su_s}\pa{R_{ij}\delta_{kt}-R_{ik}\delta_{jt}}+\frac{1}{m-2}u_sR_{si}\pa{f_k\delta_{jt}-f_j\delta_{kt}}+\frac{1}{m-2}u_s\delta_{it}\pa{f_kR_{sj}-f_jR_{sk}} \\ \nonumber &+3u_t\pa{f_ku_{ij}-f_ju_{ik}}+f_t\pa{u_ku_{ij}-u_ju_{ik}}-\pa{f_su_s}\pa{u_{ij}\delta_{kt}-u_{ik}\delta_{jt}}+\pa{f_su_s}u_i\pa{u_j\delta_{kt}-u_k\delta_{jt}}\\ \nonumber &+\frac{1}{m-1}\pa{\Delta u-\abs{\nabla u}^2}\pa{f_{kt}\delta_{ij}-f_{jt}\delta_{ik}}-5u_iu_t\pa{u_jf_k-u_kf_j}\\ \nonumber &-\frac{3}{m-1}\pa{\Delta u-\abs{\nabla u}^2}u_t\pa{f_k\delta_{ij}-f_j\delta_{ik}}-\frac{1}{m-1}\pa{\Delta u-\abs{\nabla u}^2}f_t\pa{u_k\delta_{ij}-u_j\delta_{ik}}\\ \nonumber &+\frac{\pa{f_su_s}}{m-1}\Delta u\,\pa{\delta_{ij}\delta_{kt}-\delta_{ik}\delta_{jt}}+\abs{\nabla u}^2u_i\pa{f_k\delta_{jt}-f_j\delta_{kt}}+\abs{\nabla u}^2\delta_{it}\pa{u_jf_k-u_kf_j} \\ \nonumber &-\frac{1}{m-2}u_i\pa{f_kR_{jt}-f_jR_{kt}}-\frac{1}{m-2}R_{it}\pa{u_jf_k-u_kf_j}+2u_i\pa{f_ku_{jt}-f_ju_{kt}}+2u_{it}\pa{u_jf_k-u_kf_j} \\ \nonumber &-u_su_{si}\pa{f_k\delta_{jt}-f_j\delta_{kt}}-u_s\delta_{it}\pa{f_ku_{sj}-f_ju_{sk}} - \frac{2}{m-1}u_su_{st}\pa{f_k\delta_{ij}-f_j\delta_{ik}} \\ \nonumber &-\frac{1}{m-1}f_{st}\pa{u_{ks}\delta_{ij}-u_{js}\delta_{ik}}+\frac{1}{m-1}u_sf_{st}\pa{u_k\delta_{ij}-u_j\delta_{ik}}-\frac{3}{(m-1)(m-2)}u_tf_s\pa{R_{sk}\delta_{ij}-R_{sj}\delta_{ik}} \\ \nonumber &-\frac{1}{(m-1)(m-2)}f_sR_{st}\pa{u_k\delta_{ij}-u_j\delta_{ik}}+\frac{3}{m-1}f_su_t\pa{u_{sk}\delta_{ij}-u_{sj}\delta_{ik}} \\ \nonumber &-\frac{4}{m-1}\pa{f_su_s}u_t\pa{u_k\delta_{ij}-u_j\delta_{ik}}+\frac{1}{m-1}\pa{f_su_s}\pa{u_{kt}\delta_{ij}-u_{jt}\delta_{ik}}+\frac{2}{m-1}f_su_{st}\pa{u_k\delta_{ij}-u_j\delta_{ik}} \\ \nonumber &+\frac{1}{(m-1)(m-2)}\ricc\pa{\nabla f, \nabla f}\pa{\delta_{kt}\delta_{ij}-\delta_{jt}\delta_{ik}}-\frac{1}{m-1}\hess(u)\pa{\nabla u, \nabla f}\pa{\delta_{kt}\delta_{ij}-\delta_{jt}\delta_{ik}}\\ \nonumber &+\frac{3}{(m-1)(m-2)}Su_t\pa{f_k\delta_{ij}-f_j\delta_{ik}} + \frac{1}{(m-1)(m-2)}Sf_t\pa{u_k\delta_{ij}-u_j\delta_{ik}} \\ \nonumber &-\frac{1}{(m-1)(m-2)}\pa{f_su_s} S \pa{\delta_{kt}\delta_{ij}-\delta_{jt}\delta_{ik}}.
    \end{align}

\item \textbf{Covariant derivative of a vector field and Lie derivative of the metric} (see \cite{MasRigSet}, Lemma 2.4)

 \begin{lemma}\label{LE_conformalchangeLieDeriv}
    Let $X \in \mathfrak{X}\pa{M}$ be a vector field on the Riemannian manifold $\riemanng{M}$, and let $\tilde{g}=e^{2u}g$, a conformally deformed metric. Then
    \begin{equation}\label{eq_conformalchangeLieDeriv}
      \mathcal{L}_X\tilde{g} = e^{2u}\sq{\mathcal{L}_Xg + 2 g\pa{X, \nabla u}g}.
    \end{equation}
  \end{lemma}
  \begin{proof}
  Let $\set{e_i}$, $i=1, \ldots, m$ be the frame dual to the local coframe $\set{\theta^i}$. From \eqref{2.2} we deduce that   $\tilde{e}_i=e^{-u}e_i$; moreover,
    \begin{equation}\label{XitildeXi}
    X=X^ie_i = \tilde{X}^i\tilde{e}_i,
    \end{equation}
    thus
    \begin{equation}
      \tilde{X}^i=e^u X^i.
    \end{equation}
    From the definition of covariant derivative of a vector field we have
    \begin{equation}\label{nablaXnablatildeX}
      \nabla X =X^i_k\theta^k \otimes e_i,\quad \tilde{\nabla} X = \tilde{X}^i_k\tilde{\theta}^k \otimes \tilde{e}_i,
    \end{equation}
    with $X^i_k\theta^k=(dX^i + X^j\theta^i_j)$ and $\tilde{X}^i_k\tilde{\theta}^k=(d\tilde{X}^i + \tilde{X}^j\tilde{\theta}^i_j)$.
   A computation using \eqref{nablaXnablatildeX} and \eqref{ConfTransfConnectionForms} now shows that
    \begin{equation}\label{tildeXik}
      \tilde{X}^i_k = X^i_k + \pa{X^iu_k+X^ju_j\delta_{ik}-u_iX^k},
    \end{equation}
   which implies
   \begin{equation}\label{tildeXiktildeXki}
      \tilde{X}^i_k + \tilde{X}^k_i = X^i_k +X^k_i + 2X^tu_t\delta_{ik}.
   \end{equation}
   Equation \eqref{eq_conformalchangeLieDeriv} now follows easily from \eqref{tildeXiktildeXki} and from the fact that
   \[
   \pa{\mathcal{L}_Xg}_{ij} = X^i_j+X^j_i = X_{ij}+X_{ji}.
   \]
  \end{proof}
  From equation \eqref{tildeXiktildeXki} we  deduce, tracing with respect to $i$ and $k$,
  \item \textbf{Divergence of a vector field $X\in \cinf$}:
  \begin{equation}\label{divergenzatilde}
    \tilde{\diver} X = \diver X + m g\pa{X, \nabla u}.
  \end{equation}
  Finally, from equation \eqref{tildeXik} we can obtain the following transformation law for the second covariant derivative of a vector field $X \in \mathfrak{X}\pa{M}$:
  \begin{align}\label{secondCovDerivVFExp}
  e^u\tilde{X}_{ijk} &= X_{ijk} +\pa{X_iu_{jk}-X_ju_{ik}}-\pa{X_{jk}+X_{kj}}u_i-\pa{X_iu_j-X_ju_i}u_k+\pa{X_tu_{tk}+u_tX_{tk}}\delta_{ij} \\ \nonumber &+u_t\pa{X_{it}\delta_{jk}+X_{tj}\delta_{ik}} + \pa{X_tu_t}\pa{u_j\delta_{ik}-u_i\delta_{jk}}+\abs{\nabla u}^2\pa{X_i\delta_{jk}-X_j\delta_{ik}};
  \end{align}
  in particular,
  \begin{equation}
  e^u\tilde{X}_{ttk} = X_{ttk} +m\pa{X_tu_{tk}+u_tX_{tk}}.
  \end{equation}

  \begin{rem}
  If the vector field $X$ is the gradient of a function with respect to the metric $\tilde{g}$, i.e. $X = \tilde{\nabla} f = e^{-2u}\nabla f$, it is not hard to verify that \eqref{secondCovDerivVFExp} becomes equation \eqref{thirdDerivFunctExpComp} .
  \end{rem}

\end{itemize}

\section{Commutation rules}\label{SecCommy}

%

In this section we compute commutation rules of covariant derivatives of functions, vector fields and of the geometric tensors introduced in Section \ref{SecDef}. Some of these results are well-known in the literature, some already appeared in \cite[Section 4]{CMMRGet} or in \cite{MasRigSet}, for instance, while for many of them we are not aware of any good, exhaustive reference. We collect all of them here for the sake of completeness.
We begin with
\begin{lemma}\label{LemmaCommRulesFunctions} If $f\in \cinf$ then:
\begin{align}
  f_{ij} &= f_{ji}; \label{SecondDerivFunction}\\ f_{ijk} &= f_{jik}; \label{CovDerivSecondDerivFct}\\ f_{ijk} &= f_{ikj}+f_tR_{tijk}; \label{ThirdDerivFunctionRiem}\\f_{ijk} &= f_{ikj}+f_tW_{tijk}+\frac{1}{m-2}\pa{f_tR_{tj}\delta_{ik}-f_tR_{tk}\delta_{ij}+f_jR_{ik}-f_kR_{ij}}\label{ThirdDerivFunctionWeyl}\\\nonumber &-\frac{S}{(m-1)(m-2)}\pa{f_j\delta_{ik}-f_k\delta_{ij}}; \\ f_{ijk} &= f_{ikj}+f_tW_{tijk}+\frac{1}{m-2}\pa{f_tA_{tj}\delta_{ik}-f_tA_{tk}\delta_{ij}+f_jA_{ik}-f_kA_{ij}};\label{commutatioThirdDerFunctWeilSchouten}\\f_{ijkt} &= f_{ijtk}+f_{il}R_{ljkt}+f_{jl}R_{likt};  \label{FourthDerivFunctionRiem}
  \\ f_{ijkt} &= f_{ikjt} +f_{st}R_{sijk}+f_sR_{sijk, t}; \label{ThirdDerivinfourth}
  \\ f_{ijkt} &= f_{ktij} + f_{is}R_{skjt}+f_{js}R_{skit}+f_{ks}R_{sijt}+f_{ts}R_{sijk}+f_s\pa{R_{sijk, t}-R_{skti, j}}. \label{Function12with34}
\end{align}
In particular, tracing  \eqref{ThirdDerivFunctionRiem} and \eqref{Function12with34} it follows that
\begin{align}
  f_{itt} &= f_{tti}+f_tR_{ti}; \label{TracedThirdDerivFunctionRicci}\\
  f_{ijtt} &= f_{ttij}+f_{it}R_{tj}+f_{jt}R_{ti}-2f_{st}R_{isjt}+f_t\pa{R_{tj, i}+R_{ti, j}}-f_tR_{ij, t};  \label{TracedFourthDerivFct}\\ f_{ijtt} &=f_{ttij}+f_{it}R_{tj}+f_{jt}R_{ti}-2f_{st}R_{isjt}+f_tR_{ij, t}-f_t\pa{R_{sitj, s}+R_{sjti, s}}.  \label{TracedFourthDerivFctSecondVersion}
\end{align}
\end{lemma}
\begin{rem}
  Clearly Lemma \ref{LemmaCommRulesFunctions} still works if $f$ is at least of class $C^4(M)$.
\end{rem}
\begin{proof}
Let $df = f_i\theta^i$. Differentiating and using the structure equations we get
\begin{align*}
0&=df_i \wedge \theta^i + f_i d\theta^i = (f_{ij}\theta^j +
f_k\theta^k_i)\wedge\theta^i - f_i\theta^i_k \wedge \theta^k \\&=
f_{ij}\theta^j \wedge \theta^i\\&= \frac 12 \pa{f_{ij}-f_{ji}}\theta^j \wedge \theta^i,
\end{align*}
thus
\[
0= \sum_{1 \leq j <i \leq m}(f_{ij}-f_{ji})\theta^j \wedge \theta^i;
\]
since $\set{\theta^j \wedge \theta^i}$  $\pa{1 \leq j < i \leq m}$
is a basis for the $2$-forms we get equation \eqref{SecondDerivFunction}. Equation \eqref{CovDerivSecondDerivFct} follows taking the covariant derivative of \eqref{SecondDerivFunction}. By definition of covariant derivative
\begin{equation}\label{1_derivate_terze}
  f_{ijk}\theta^k = df_{ij} - f_{kj}\theta^k_i - f_{ik}\theta^k_j.
\end{equation}

Differentiating equation \eqref{HessianComponents} and using the structure
equations we get
\begin{align*}
  df_{ik}\wedge\theta^k - f_{ij}\theta^j_k \wedge\theta^k &= - df_t
\wedge \theta^t_i + f_k\theta^k_t \wedge \theta^t_i - f_k\Theta^k_i
= \\ &=-(f_{tk}\theta^k + f_k\theta^k_t) \wedge\theta^t_i +
f_k\theta^k_t \wedge \theta^t_i - \frac{1}{2}f_kR^k_{ijt} \theta^j
\wedge \theta^t,
\end{align*}
thus
\[
(df_{ik} - f_{tk}\theta^t_i - f_{it}\theta^t_k)\wedge \theta^k
=-\frac{1}{2}f_t R^t_{ijk}\theta^j \wedge \theta^k,
\]
and, by \eqref{1_derivate_terze},
\[
f_{ikj}\theta^j \wedge \theta^k = -\frac{1}{2} f_t R^t_{ijk}
\theta^j \wedge \theta^k.
\]
Skew-symmetrizing we get
\[
\frac{1}{2}(f_{ikj}-f_{ijk})\theta^j \wedge \theta^k =
-\frac{1}{2}f_tR^t_{ijk}\theta^j \wedge\theta^k,
\]
that is \eqref{ThirdDerivFunctionRiem}. Equations \eqref{ThirdDerivFunctionWeyl} and \eqref{commutatioThirdDerFunctWeilSchouten} follow easily from \eqref{ThirdDerivFunctionRiem}, using the definitions of the Weyl tensor and of the Schouten tensor (see Section 2). To prove \eqref{FourthDerivFunctionRiem} we start from \eqref{ThirdDerivFunctionRiem} and we take the covariant derivative to deduce
    \begin{equation}\label{covDerivY3rdComm}
      f_{ijkt} - f_{ikjt} = f_{st}R_{sijk} + f_sR_{sijk, \,t}.
    \end{equation}
    Differentiating both sides of \eqref{1_derivate_terze}, using the structure equations and \eqref{1_derivate_terze} itself, we arrive at
    \[
    f_{ijkl} \theta^l \wedge \theta^k = -\frac 12 \pa{f_{tj}R_{tilk} + f_{it}R_{tjlk}} \theta^l \wedge \theta^k,
    \]
    from which, interchanging $k$ and $l$ and adding, we have the thesis. Equation \eqref{Function12with34} now follows using all the previous relations, starting from \eqref{ThirdDerivinfourth} .
\end{proof}

For the components of a vector field and for their covariant derivatives the commutation relations are similar to the ones proved for functions in Lemma \ref{LemmaCommRulesFunctions}; in particular we have the following
\begin{lemma}[Lemma 2.1 in \cite{MasRigRim}]\label{CommutationsForVectorFields}
Let $X \in \mathfrak{X}(M)$ be a vector field. Then we have
\begin{eqnarray}
X_{ijk}-X_{ikj} &=& X_tR_{tijk}; \\ X_{ijkl}-X_{ikjl} &=& R_{tijk}X_{tl}+R_{tijk, l}X_t; \\ X_{ijkl}-X_{ijlk} &=& R_{tikl}X_{tj}+R_{tjkl}X_{it}.
\end{eqnarray}
\end{lemma}

Concerning the Riemann curvature tensor, we begin with the classical Bianchi identities, that in our formalism become
\begin{align}
  &R_{ijkt}+R_{itjk}+R_{iktj}=0  \quad \text{(the First Bianchi Identity)};\label{FirstBianchiRiem}\\ &R_{ijkt, l}+R_{ijlk, t}+R_{ijtl, k}=0  \quad \text{(the Second Bianchi Identity)}. \label{SecondBianchiRiem}
\end{align}
For the second and third derivatives we prove
\begin{lemma}\label{LemmaSTRiemann}
\begin{align}
   &R_{ijkt, lr}-R_{ijkt, rl} = R_{sjkt}R_{silr}+R_{iskt}R_{sjlr}+R_{ijst}R_{sklr}+R_{ijks}R_{stlr}; \label{SecondDerivRiem}\\ &R_{ijkt, lrs}-R_{ijkt, lsr} = R_{vjkt, l}R_{virs}+R_{ivkt, l}R_{vjrs}+R_{ijvt, l}R_{vkrs}+R_{ijkv, l}R_{vtrs}+R_{ijkt, v}R_{vlrs}. \label{ThirdDerivRiem}
\end{align}
\end{lemma}
\begin{proof}
 By definition of covariant derivative we have
 \begin{equation}\label{firstCovDerivRiemComp}
   R_{ijkt, l}\theta^l = d R_{ijkt}-R_{ljkt}\theta^l_i-R_{ilkt}\theta^l_j-R_{ijlt}\theta^l_k-R_{ijkl}\theta^l_t
 \end{equation}
 and
  \begin{equation}\label{secondCovDerivRiemComp}
   R_{ijkt, lr}\theta^r = d R_{ijkt, l}-R_{ljkt, l}\theta^r_i-R_{irkt, l}\theta^r_j-R_{ijrt, l}\theta^r_k-R_{ijkr, l}\theta^r_t-R_{ijkt, r}\theta^r_l.
 \end{equation}
 Differentiating equation \eqref{firstCovDerivRiemComp} and using the first structure equations we get
 \begin{align}
   d R_{ijkt, s}\wedge \theta^s-R_{ijkt, l}\theta^l_s\wedge\theta^s &= -d R_{ljkt}\wedge\theta^l_i+R_{ljkt}\pa{\theta^l_s\wedge\theta^s_i-\Theta^l_i}-d R_{ilkt}\wedge\theta^l_j+R_{ilkt}\pa{\theta^l_s\wedge\theta^s_j-\Theta^l_j}\\ \nonumber &-d R_{ijlt}\wedge\theta^l_i+R_{ijlt}\pa{\theta^l_s\wedge\theta^s_k-\Theta^l_k}-d R_{ijkl}\wedge\theta^l_i+R_{ijkl}\pa{\theta^l_s\wedge\theta^s_t-\Theta^l_t}.
 \end{align}
 Now we repeatedly use \eqref{secondCovDerivRiemComp} and \eqref{1_def_forme_curvatura} into the previous relation; after some manipulations we arrive at
 \begin{align*}
   \pa{d R_{ijkt, s}-R_{ljkt, s}\theta^l_i-R_{ilkt, s}\theta^l_j-R_{ijlt, s}\theta^l_k-R_{ijkl, s}\theta^l_t-R_{ijkt, l}\theta^l_s}\wedge\theta^s &= -\frac 12\left(R_{ljkt}R_{lirs}+R_{ilkt}R_{ljrs}\right. \\ \nonumber &\left.+R_{ijlt}R_{lkrs}+R_{ijkl}R_{ltrs}\right)\theta^r\wedge\theta^s.
 \end{align*}
 Renaming indices and skew-symmetrizing the left hand side, which is precisely $R_{ijkt, sr}\theta^r\wedge\theta^s$, we obtain \eqref{SecondDerivRiem}. A similar computation shows the validity of  \eqref{ThirdDerivRiem}.
\end{proof}

For the Ricci and the Schouten tensors we have the following
\begin{lemma}\label{LemmaFSTDerivRicci}
\begin{align}
  &R_{ij, k}-R_{ik, j} = -R_{tijk, t} = R_{tikj, t} \\ &R_{ij, kt}-R_{ij, tk}=R_{likt}R_{lj}+R_{ljkt}R_{li} \\ &R_{ij, ktl}-R_{ij, klt} = R_{sj, k}R_{sitl} +R_{is, k}R_{sjtl} + R_{ij, s}R_{sktl}.
\end{align}
\end{lemma}
\begin{proof}
  The previous relations follow easily tracing equations \eqref{SecondBianchiRiem}, \eqref{SecondDerivRiem} and \eqref{ThirdDerivRiem}, respectively.
\end{proof}

A simple computation using the definition of the Schouten tensor, Lemma \ref{LemmaFSTDerivRicci} and equations \eqref{SecondDerivFunction} and \eqref{ThirdDerivFunctionRiem} applied to the scalar curvature shows the validity of
\begin{lemma}
\begin{align}
  &A_{ij, k}-A_{ik, j} = C_{ijk} = \pa{\frac{m-2}{m-3}}W_{tikj, t} \\ &A_{ij, kt}-A_{ij, tk}=R_{likt}A_{lj}+R_{ljkt}A_{li} \\ &A_{ij, ktl}-A_{ij, klt} = A_{sj, k}R_{sitl} +A_{is, k}R_{sjtl} + A_{ij, s}R_{sktl}.
\end{align}
\end{lemma}

A direct consequence of the definition of the Weyl tensor and of the First Bianchi identity for the Riemann curvature tensor is the First Bianchi identity for $W$:
\begin{align}
  &W_{ijkt}+W_{itjk}+W_{iktj}=0.
\end{align}
As far as the first derivatives of $W$ are concerned, we have
\begin{lemma}\label{lemma_fake2ndBianchiWeyl}
  \begin{equation}\label{fake2ndBianchiWeyl}
    W_{ijkt, l}+W_{ijlk, t}+W_{ijtl, k}=\frac{1}{m-2}\pa{C_{itl}\delta_{jk}+C_{ilk}\delta_{jt}+C_{ikt}\delta_{jl}-C_{jtl}\delta_{ik}-C_{jlk}\delta_{it}-C_{jkt}\delta_{il}}.
  \end{equation}
\end{lemma}
\begin{proof}
  We start by taking the covariant derivative of \eqref{Riemann_Weyl}:
  \begin{equation}\label{CovDerivRiemann_Weyl}
    R_{ijkt, l} = W_{ijkt, l} + \frac{1}{m-2}\pa{R_{ik, l}\delta_{jt}-R_{it, l}\delta_{jk}+R_{jt, l}\delta_{ik}-R_{jk, l}\delta_{it}}-\frac{S_l}{(m-1)(m-2)}\pa{\delta_{ik}\delta_{jt}-\delta_{it}\delta_{jk}}.
  \end{equation}
  Permuting cyclically the last three indices, summing up and using \eqref{FirstBianchiRiem} we deduce
  \begin{align*}
    -\pa{W_{ijkt, l}+W_{ijlk, t}+W_{ijtl, k}} &=\frac{1}{m-2}\sq{\pa{R_{ik, l}-R_{il, k}}\delta_{jt}+\pa{R_{il, t}-R_{it, l}}\delta_{jk}+\pa{R_{it, k}-R_{ik, t}}\delta_{jl}} \\ &-\frac{1}{m-2}\sq{\pa{R_{jk, l}-R_{jl, k}}\delta_{it}+\pa{R_{jl, t}-R_{jt, l}}\delta_{ik}+\pa{R_{jt, k}-R_{jk, t}}\delta_{il}}\\ &-\frac{1}{(m-1)(m-2)}\sq{S_l\pa{\delta_{ik}\delta_{jt}-\delta_{it}\delta_{jk}}+S_t\pa{\delta_{il}\delta_{jk}-\delta_{ik}\delta_{jl}}+S_k\pa{\delta_{it}\delta_{jl}-\delta_{il}\delta_{jt}}}.
  \end{align*}
  Using the fact that $R_{ij, k}-R_{ik, j} = C_{ijk}+\frac{1}{2(m-1)}\pa{S_k\delta_{ij}-S_{j}\delta_{ik}}$, after some manipulation we get \eqref{fake2ndBianchiWeyl}.
\end{proof}
For the second and third derivatives of $W$, a computation similar to the one used in the proof of Lemma \ref{LemmaSTRiemann} shows that
\begin{lemma}
\begin{align}
  &W_{ijkl, st}-W_{ijkl, ts} = W_{rjkl}R_{rist}+W_{irkl}R_{rjst}+W_{ijrl}R_{rkst}+W_{ijkr}R_{rlst}; \label{SecondDerivWeylusingRiem}\\ &W_{ijkl, trs}-W_{ijkl, tsr} = W_{vjkl, t}R_{virs}+W_{ivkl, t}R_{vjrs}+W_{ijvl, t}R_{vkrs}+W_{ijkv, t}R_{vlrs}+W_{ijkl, v}R_{vtrs}. \label{ThirdDerivWeylusingRiem}
\end{align}
\end{lemma}
Using the definition of the Weyl tensor in equation \eqref{SecondDerivWeylusingRiem} we obtain
\begin{lemma}
\begin{align}
  W_{ijkl, st}-W_{ijkl, ts} &=W_{rjkl}W_{rist}+W_{irkl}W_{rjst}+W_{ijrl}W_{rkst}+W_{ijkr}W_{rlst} \\ \nonumber &+\frac{1}{m-2}\left[W_{rjkl}\pa{R_{rs}\delta_{it}-R_{rt}\delta_{is}+R_{it}\delta_{rs}-R_{is}\delta_{rt}}\right.\\\nonumber &\qquad\qquad +W_{irkl}\pa{R_{rs}\delta_{jt}-R_{rt}\delta_{js}+R_{jt}\delta_{rs}-R_{js}\delta_{rt}}\\\nonumber
   & \qquad\qquad +W_{ijrl}\pa{R_{rs}\delta_{kt}-R_{rt}\delta_{ks}+R_{kt}\delta_{rs}-R_{ks}\delta_{rt}} \\\nonumber &\left.\qquad\qquad+W_{ijkr}\pa{R_{rs}\delta_{lt}-R_{rt}\delta_{ls}+R_{lt}\delta_{rs}-R_{ls}\delta_{rt}}\right]
  \\\nonumber &-\frac{S}{(m-1)(m-2)}\left[W_{rjkl}\pa{\delta_{rs}\delta_{it}-\delta_{rt}\delta_{is}}+W_{irkl}\pa{\delta_{rs}\delta_{jt}-\delta_{rt}\delta_{js}}\right.\\ \nonumber &\qquad\qquad\qquad\qquad\left.+W_{ijrl}\pa{\delta_{rs}\delta_{kt}-\delta_{rt}\delta_{ks}}+W_{ijkr}\pa{\delta_{rs}\delta_{lt}-\delta_{rt}\delta_{ls}}\right].
\end{align}
\end{lemma}
Tracing the previous relation  we also get
\begin{align}
  W_{tjkl, st}-W_{tjkl, ts} &= R_{st}W_{tjkl}+W_{trkl}W_{rjst}+W_{tjrl}W_{rkst}+W_{tjkr}W_{rlst}\\ \nonumber &+\frac{1}{m-2}\pa{R_{tr}W_{tjrk}\delta_{ls}-R_{tr}W_{tjrl}\delta_{ks}} \\ \nonumber &+\frac{1}{m-2}\pa{R_{tk}W_{tjsl}+R_{tl}W_{tjks}+R_{tj}W_{tskl}}.
\end{align}
Using the definition of the Weyl tensor in equation \eqref{ThirdDerivWeylusingRiem} we obtain
\begin{lemma}
\begin{align}
  W_{ijkl, trs}-W_{ijkl, tsr} &=W_{vjkl, t}W_{virs}+W_{ivkl, t}W_{vjrs}+W_{ijvl, t}W_{vkrs}+W_{ijkv, t}W_{vlrs}+W_{ijkl, v}W_{vtrs} \\ \nonumber &+\frac{1}{m-2}\left[W_{vjkl, t}\pa{R_{vr}\delta_{is}-R_{vs}\delta_{ir}+R_{is}\delta_{vr}-R_{ir}\delta_{vs}}\right.\\\nonumber &\qquad\qquad +W_{ivkl, t}\pa{R_{vr}\delta_{js}-R_{vs}\delta_{jr}+R_{js}\delta_{vr}-R_{jr}\delta_{vs}}\\\nonumber &\qquad\qquad +W_{ijvl, t}\pa{R_{vr}\delta_{ks}-R_{vs}\delta_{kr}+R_{ks}\delta_{vr}-R_{kr}\delta_{vs}}
  \\\nonumber
   & \qquad\qquad +W_{ijkv, t}\pa{R_{vr}\delta_{ls}-R_{vs}\delta_{lr}+R_{ls}\delta_{vr}-R_{lr}\delta_{vs}} \\\nonumber &\left.\qquad\qquad+W_{ijkl, v}\pa{R_{vr}\delta_{ts}-R_{vs}\delta_{tr}+R_{ts}\delta_{vr}-R_{tr}\delta_{vs}}\right]
  \\ \nonumber &-\frac{S}{(m-1)(m-2)}\left[W_{vjkl, t}\pa{\delta_{vr}\delta_{is}-\delta_{vs}\delta_{ir}}+W_{ivkl, t}\pa{\delta_{vr}\delta_{js}-\delta_{vs}\delta_{jr}}\right.\\ \nonumber &\qquad\qquad\qquad\qquad W_{ijvl, t}\pa{\delta_{vr}\delta_{ks}-\delta_{vs}\delta_{kr}}+W_{ijkv, t}\pa{\delta_{vr}\delta_{ls}-\delta_{vs}\delta_{lr}} \\ \nonumber &\qquad\qquad\qquad\qquad\left.+W_{ijkl, v}\pa{\delta_{vr}\delta_{ts}-\delta_{vs}\delta_{tr}}\right].
\end{align}
\end{lemma}

The First Bianchi Identities for the Weyl tensor immediately imply
\begin{equation}\label{PermutCiclCotton}
  C_{ijk}+C_{jki}+C_{kij}=0.
\end{equation}
From the definition of the Cotton tensor we also deduce
\begin{equation}
  C_{ijk, t} = A_{ij,kt}-A_{ik, jt}=R_{ij, kt}-R_{ik, jt}-\frac{1}{2(m-1)}\pa{S_{kt}\delta_{ij}-S_{jt}\delta_{ik}};
\end{equation}
since, by Lemma \ref{LemmaFSTDerivRicci} and Schur's identity $S_i = \frac 12 R_{ik, k}$,
\begin{equation}
  R_{ik, jk} = R_{ik, kj}+R_{tijk}R_{tk}+R_{tkjk}R_{ti}=\frac 12 S_{ij}-R_{tk}R_{itjk}+R_{it}R_{tj},
\end{equation}
we obtain the following expression for the divergence of the Cotton tensor:
\begin{equation}\label{DiverCotton}
  C_{ijk, k}= R_{ij, kk}-\frac{m-2}{2(m-1)} S_{ij}+R_{tk}R_{itjk}-R_{it}R_{tj}-\frac{1}{2(m-1)}\Delta S\delta_{ij}.
\end{equation}
The previous relation also shows that
\begin{equation}\label{SymmDivCotton}
C_{ijk, k}=C_{jik, k},
\end{equation}
thus confirming the symmetry of the Bach tensor, see \eqref{def_Bach_comp}.

Taking the covariant derivative of \eqref{PermutCiclCotton} and using \eqref{SymmDivCotton} we can also deduce that
\begin{equation}\label{NullDiverCotton}
  C_{kij, k}=0.
\end{equation}

\section{Some useful relations for Ricci solitons}\label{sec5}
The aim of this short section is to recall a number of useful relations, valid on every Ricci soliton, that have been consistently exploited in the literature to obtain several well known results.

First we have  (see also \cite{MasRigRim}, Lemma 2.2 and Lemma 2.3, \cite{ELnM}):

\begin{proposition} Let $(M, g, X)$ be a generic Ricci soliton structure on $\varrg$. Then the following identities hold:
\begin{eqnarray}
&\label{eq1} R_{ij} +\frac{1}{2}(X_{ij}+X_{ji})=\lambda \delta_{ij};\\
&\label{eq2}S +  \diver X=m\lambda; \\
&\label{eq3}S_k = -X_{iik}; \\
&\label{eq4} R_{tj}X_{t}=-X_{ktt};\\
&\label{eq5} R_{ij,k}-R_{ik,j}=-\frac{1}{2}R_{lijk}X_{l}+\frac{1}{2}(X_{kij}-X_{jik}); \\
&\label{eq6} R_{ij,k}-R_{kj,i}=\frac{1}{2}R_{ljki}X_{l}+\frac{1}{2}(X_{kji}-X_{ijk}); \\
&\label{scalGen} \frac{1}{2}\Delta S = \frac{1}{2}g\pa{X, \nabla S} +\lambda S - \abs{\ricc}^2.
\end{eqnarray}
 If $X=\nabla f$ for some $f\in\cinf$ then
 \begin{eqnarray}
 &\label{eq1g} R_{ij} +f_{ij}=\lambda \delta_{ij};\\
&\label{eq2g}S +  \Delta f=m\lambda; \\
&\label{eq3g}S_k = 2f_tR_{tk}; \\
&\label{eq6g} R_{ij,k}-R_{kj,i}= -f_tR_{tijk}; \\
&\label{HamiltonId} S + \abs{\nabla f}^2-2\lambda f = C, \quad C \in \erre; \\
&\label{scalGrad} \frac{1}{2}\Delta S = \frac{1}{2}g\pa{\nabla f, \nabla S} +\lambda S - \abs{\ricc}^2.
\end{eqnarray}
\end{proposition}

From the work of Cao and Chen (see \cite{CaoChen}, Lemma 3.1 and  equation (4.1); see also \cite{CMMRGet}), we have the validity of the following integrability conditions:

\begin{theorem}\label{caochenth} If $\pa{M, g, f}$ is a gradient Ricci soliton with potential function $f$, then the Cotton tensor, the Weyl tensor, the Bach tensor, the potential and the tensor $D$ satisfy the conditions:
  \begin{align}
     &C_{ijk}+f_t W_{tijk} = D_{ijk}, \label{firstCaoChen}\\     &B_{ij} = \frac{1}{m-2}\sq{D_{ijk, k}+\pa{\frac{m-3}{m-2}}f_tC_{jit}}. \label{secondCaoChen}
  \end{align}
\end{theorem}

\begin{rem} From \eqref{firstCaoChen} we deduce
\begin{equation}
f_tC_{tij} = f_tD_{tij}.
\end{equation}
\end{rem}

%

Moreover, letting $[ijk]$ denote a summed cyclic permutation of $i, j, k$ (for example $T_{[ijk]} = T_{ijk}+T_{jki}+T_{kij}$), a long but straightforward calculation shows that for the tensor $D$ the following holds:
\begin{lemma}
Let $(M, g, \nabla f)$ be a gradient Ricci soliton structure on $\varrg$. Then the following identities hold:
\begin{equation}
D_{[ijk]}=0;
\end{equation}
\begin{equation}
D_{i[jk, t]}= \frac{1}{m-2}\sq{f_l\pa{C_{lkt}\delta_{ij}+C_{ltj}\delta_{ik}+C_{ljk}\delta_{it}}-\pa{f_jC_{ikt}+f_kC_{itj}+f_tC_{ijk}}}
\end{equation}
\begin{equation}
D_{i[jk, t]} = \frac{1}{m-2}\sq{f_l\pa{D_{lkt}\delta_{ij}+D_{ltj}\delta_{ik}+D_{ljk}\delta_{it}}-f_j\pa{D_{ikt}-f_sW_{sikt}}-f_k\pa{D_{itj}-f_sW_{sitj}}-f_t\pa{D_{ijk}-f_sW_{sijk}}}.
\end{equation}
\begin{equation}
C_{i[jk, t]} =  R_{sj}W_{sikt}+R_{sk}W_{sitj}+R_{st}W_{sijk}.
\end{equation}
\begin{align}
D_{i[jk, t]}  &= \frac{m-6}{2(m-3)}\pa{R_{sj}W_{sikt}+R_{sk}W_{sitj}+R_{st}W_{sijk}-C_{i[jk, t]} } \\ \nonumber &+\frac{1}{m-2}\sq{f_l\pa{C_{lkt}\delta_{ij}+C_{ltj}\delta_{ik}+C_{ljk}\delta_{it}}-\pa{f_jC_{ikt}+f_kC_{itj}+f_tC_{ijk}}}.
\end{align}

\end{lemma}

\section{Conformally Einstein metrics}\label{SecConfy}

In this short section we first recall the definition of a conformally Einstein manifold; then we present the integrability conditions of Gover and Nurowski and we prove equation \eqref{PR_CE_LaplacianScalarEq}, which relates the Laplacian of the scalar curvature of a conformally Einstein manifold to $u$ (the exponent of the stretching factor) and its covariant derivatives.

\begin{defi}
  A Riemannian manifold $\varrg$ is said to be \emph{conformally Einstein} if there exists a conformal change of the metric  $\tilde{g}=e^{2u} g, \, u\in\cinf$, such that $\pa{M, \tilde{g}}$ is Einstein, i.e. \begin{equation}\label{CE_global_tilde}
  \widetilde{\ricc} = \frac{\tilde{S}}{m}\tilde{g} = \lambda \tilde{g}, \qquad \lambda \in \erre.
  \end{equation}

\end{defi}
Since in an orthonormal frame \eqref{CE_global_tilde} becomes
 \begin{equation}\label{CE_comp_tilde}
   \tilde{R}_{ij}=\frac{\tilde{S}}{m}\delta_{ij} = \lambda \delta_{ij},
 \end{equation}
 using equations \eqref{RicciexpComponents} and \eqref{scalarExp} we can easily deduce that $\varrg$ is conformally Einstein if and only if there exists a solution $u\in \cinf$ of the equation
\begin{equation}\label{CE_comp_Riccii}
  R_{ij} -(m-2)u_{ij} + (m-2)u_iu_j = \frac{1}{m}\sq{S-(m-2)\Delta u+(m-2)\abs{\nabla u}^2}\delta_{ij},
\end{equation}
with
\begin{equation}\label{CE_tracedlambda}
S-2(m-1)\Delta u -(m-1)(m-2)\abs{\nabla u}^2 = \lambda m e^{2u}.
\end{equation}

Equation \eqref{CE_comp_Riccii} can be also written in terms of the Schouten tensor as
\begin{equation}\label{CE_comp_Schouten}
  A_{ij} -(m-2)u_{ij} + (m-2)u_iu_j = \frac{1}{m}\sq{\frac{(m-2)S}{2(m-1)}-(m-2)\Delta u+(m-2)\abs{\nabla u}^2}\delta_{ij}.
\end{equation}
\begin{rem}
Note that equation \eqref{CE_tracedlambda} is just the trace of \eqref{CE_comp_Riccii}. The system \eqref{CE_comp_Riccii}-\eqref{CE_tracedlambda} is equivalent to the single equation
\begin{equation}\label{CE_singleEq}
R_{ij} -(m-2)u_{ij} + (m-2)u_iu_j = \sq{\Delta u+(m-2)\abs{\nabla u}^2+\lambda e^{2u}}\delta_{ij}.
\end{equation}

\end{rem}

\begin{rem}
  The global version of equation \eqref{CE_comp_Riccii} is
  \begin{equation}
    \ricc -(m-2)\hess\pa{u} + (m-2)du\otimes du = \frac{1}{m}\sq{S-(m-2)\Delta u+(m-2)\abs{\nabla u}^2}g.
  \end{equation}
\end{rem}

We have the following  proposition, reported in Gover and Nurowski (\cite{GoverNurowski}, Proposition 2.1), which describes the integrability conditions of conformally Einstein metrics:
\begin{proposition}\label{Prop_Gover_Nurowski2.1}
  If $\varrg$ is a conformally Einstein Riemannian manifold, then the Cotton tensor, the Weyl tensor, the Bach tensor and the exponent $u$ of the stretching factor satisfy the conditions:
 \begin{align}
   &C_{ijk} -(m-2)u_tW_{tijk}=0, \label{FirstCond_GN}\\ &B_{ij}-(m-4)u_tu_kW_{itjk}=0\label{SecondCond_GN}.
 \end{align}
\end{proposition}
The proof of \eqref{FirstCond_GN} starts from the covariant derivative of \eqref{CE_comp_Schouten}; one then skew-symmetrizes, traces and rearranges (after a lot of simple but long calculations). Taking the divergence of  \eqref{FirstCond_GN}, using the definition of the Bach tensor \eqref{BachExpComp} and equation \eqref{CE_comp_Riccii} one gets \eqref{SecondCond_GN}. We do not provide the details here since we shall consider later a general computation including this proposition as a particular case (see \ref{TH_FirstConditionCGRS} and \ref{CGRS_TH_SecCond}).
The interesting fact is that \eqref{FirstCond_GN}, \eqref{SecondCond_GN} and \eqref{firstCaoChen}, \eqref{secondCaoChen} are strictly related, as it will become apparent in a short while.

Taking the covariant derivative of \eqref{CE_tracedlambda} and using \eqref{CE_comp_Riccii} to substitute the Hessian of $u$ we deduce the interesting relation
\begin{equation}\label{CE_nablaDeltau}
u_{ttk} = \frac{S_k}{2(m-1)} - u_tR_{tk} - \frac{1}{m(m-1)}S u_k +\pa{\frac{m+2}{m}} \Delta u\, u_k + \pa{\frac{m-2}{m}}\abs{\nabla u}^2 u_k,
\end{equation}
which implies
\begin{equation}\label{CE_gnablaunabladeltau}
g\pa{\nabla u, \nabla \Delta u} = \frac{1}{2(m-1)}g\pa{\nabla S, \nabla u} - \ricc\pa{\nabla u, \nabla u}-\frac{1}{m(m-1)}S\abs{\nabla u}^2+\pa{\frac{m+2}{m}} \Delta u\abs{\nabla u}^2+ \pa{\frac{m-2}{m}}\abs{\nabla u}^4.
\end{equation}

Now we use the fact that $\tilde{S}$ is constant and thus
\begin{equation}\label{CE_deltatildeSzero}
e^{4u}\tilde{\Delta}\tilde{S}=0;
\end{equation}
Moreover, we observe that, from equation \eqref{CE_comp_Riccii},
\begin{equation}\label{CE_ricchess}
\ricc\pa{\nabla u, \nabla u} - (m-2)\hess(u)\pa{\nabla u, \nabla u} = \frac 1m \abs{\nabla u}^2\sq{S-(m-2)\Delta u-(m-1)(m-2)\abs{\nabla u}^2}.
\end{equation}
Using \eqref{CE_deltatildeSzero}, \eqref{CE_ricchess} and \eqref{CE_gnablaunabladeltau} in \eqref{LaplacianscalarExp} and simplifying we deduce the following
\begin{proposition}\label{PR_CE_LaplacianScalarEq}
Let $\varrg$ be a conformally Einstein manifold. Then
\begin{align}\label{CE_LaplacianScalarEq}
\frac12\sq{\Delta S-(m-2)g\pa{\nabla S, \nabla u}} &= (m-1)\Delta^2u+(m-1)(m-2)\abs{\hess(u)}^2+S\Delta u-2(m-1)\pa{\Delta u}^2\\ \nonumber &+\pa{\frac{m+2}{m}}\abs{\nabla u}^2\sq{S-2(m-1)\Delta u-(m-1)(m-2)\abs{\nabla u}^2}.
\end{align}
\end{proposition}
\begin{rem}
Equation \eqref{CE_LaplacianScalarEq} can also be obtained by taking the Laplacian of both sides of \eqref{CE_tracedlambda}, using the divergence of equation \eqref{CE_nablaDeltau} and the classical Bochner-Weitzenb\"ock formula (see e.g. \cite{Besse}).
\end{rem}
\begin{rem}
Since, by equation \eqref{CE_tracedlambda},
\[
\pa{\frac{m+2}{m}}\abs{\nabla u}^2\sq{S-2(m-1)\Delta u-(m-1)(m-2)\abs{\nabla u}^2}= (m+2)\lambda e^{2u}\abs{\nabla u}^2,
\]
equation \eqref{CE_LaplacianScalarEq} can also be written as
\begin{align}\label{CE_LaplacianScalarEqwithLambda}
\frac12\sq{\Delta S-(m-2)g\pa{\nabla S, \nabla u}} &= S\Delta u-2(m-1)\pa{\Delta u}^2 +(m-1)\Delta^2u+(m-1)(m-2)\abs{\hess(u)}^2\\\nonumber &+(m+2)\lambda e^{2u}\abs{\nabla u}^2.
\end{align}
\end{rem}
\begin{rem}
If we take $u=\log v^{\frac{2}{m-2}}$, for some $v\in\cinf$, $v>0$, equation \eqref{CE_tracedlambda} becomes  the classical Yamabe equation
\[
\frac{4(m-1)}{m-2}\Delta v - Sv+\tilde{S}v^{\frac{m+2}{m-2}}=0,
\]
while equation \eqref{CE_comp_Riccii} becomes
\[
R_{ij}-2\frac{v_{ij}}{v}+\frac{2m}{m-2}\frac{v_iv_j}{v^2} = \frac 1m \sq{S-2\frac{\Delta v}{v}+\frac{2m}{m-2}\frac{\abs{\nabla v}^2}{v^2}}\delta_{ij}.
\]

\end{rem}


%
%

\section{Conformal gradient Ricci solitons}\label{sec_ConfGradRS}

In this section we introduce the notion of a conformal gradient Ricci soliton, inspired by the two particular cases of Ricci solitons and conformally Einstein metrics, in order to create a link between them.

\begin{defi}
  A Riemannian manifold $\varrg$ is said to be a \emph{conformal gradient Ricci soliton} if there exist a conformal change of the metric  $\tilde{g}=e^{2u} g, \, u\in\cinf$, a function $f\in \cinf$ and a constant $\lambda \in \erre$ such that $\pa{M, \tilde{g}}$ is a gradient Ricci soliton, i.e.
  \begin{equation}\label{CGRS_global_tilde}
  \widetilde{\ricc} +\widetilde{\operatorname{Hess}}(f)= \lambda\tilde{g}.
  \end{equation}
\end{defi}

In terms of the geometry of the manifold $\varrg$,  \eqref{CGRS_global_tilde} leads to the following
\begin{lemma}
$\varrg$ is a conformal gradient Ricci soliton if and only if there exist  $u\in\cinf$, a function $f\in \cinf$ and a constant $\lambda \in \erre$ such that
\begin{align}\label{Eq_CGRSGlobal}
\ricc - (m-2)\hess\pa{u}+&(m-2)du\otimes du +\hess\pa{f}-\pa{df\otimes du+du\otimes df} = \\ \nonumber &\frac{1}{m}\sq{S-(m-2)\pa{\Delta u-\abs{\nabla u}^{2}}+\Delta f-2g\pa{\nabla f, \nabla u}}g
\end{align}
and
\begin{align}\label{Eq_CGRSGlobalTraced}
S-2(m-1)\Delta u - (m-1)(m-2)\abs{\nabla u}^{2}+\Delta f +(m-2)g\pa{\nabla f, \nabla u} = m\lambda e^{2u}.
\end{align}

\end{lemma}

\begin{proof}
In an orthonormal frame \eqref{CGRS_global_tilde} becomes
 \begin{equation}\label{CGRS_comp_tilde}
   \tilde{R_{ij}}+\tilde{f_{ij}}=\lambda \delta_{ij},
 \end{equation}
 while tracing \eqref{CGRS_global_tilde} we deduce that
  \begin{equation}\label{Eq_lambda_tildeS_tildeDeltaf}
   m\lambda=\tilde{S} +\tilde{\Delta}f.
 \end{equation}
 Multiplying both sides of \eqref{Eq_lambda_tildeS_tildeDeltaf} through $e^{2u}$ and using \eqref{scalarExp} and \eqref{LaplacianExpComp} we get \eqref{Eq_CGRSGlobalTraced}; multiplying both sides of \eqref{CGRS_comp_tilde} by $e^{2u}$, using \eqref{RicciexpComponents}, \eqref{HessianExpComp} and \eqref{Eq_CGRSGlobalTraced} we deduce \begin{equation}\label{CGRS_comp_Ricci}
R_{ij} -(m-2)u_{ij} + (m-2)u_iu_j +f_{ij} -\pa{f_{i}u_{j}+f_{j}u_{i}}=\frac{1}{m}\sq{S-(m-2)\pa{\Delta u-\abs{\nabla u}^{2}}+\Delta f-2g\pa{\nabla f, \nabla u}}\delta_{ij},
\end{equation}
that is \eqref{Eq_CGRSGlobal}.
 \end{proof}


Note that equation \eqref{CGRS_comp_Ricci} can be written, using the Schouten tensor, as
\begin{equation}\label{CGRS_comp_Schouten}
A_{ij} -(m-2)u_{ij} + (m-2)u_iu_j +f_{ij} -\pa{f_{i}u_{j}+f_{j}u_{i}}=\frac{1}{m}\sq{\frac{m-2}{2(m-1)}S-(m-2)\pa{\Delta u-\abs{\nabla u}^{2}}+\Delta f-2g\pa{\nabla f, \nabla u}}\delta_{ij}.
\end{equation}

For a conformal gradient Ricci soliton we define the tensor $D^{\pa{u, f}}$ as follows:
\begin{align}\label{tensorD_u_f_bestVersion}
D^{\pa{u, f}}_{ijk} &= \frac{1}{m-2}\pa{f_kR_{ij}-f_jR_{ik}}+\frac{1}{(m-1)(m-2)}f_t\pa{R_{tk}\delta_{ij}-R_{tj}\delta_{ik}}-\frac{S}{(m-1)(m-2)}\pa{f_k\delta_{ij}-f_j \delta_{ik}} \\ \nonumber &+\frac{\Delta u}{m-1}\pa{f_k\delta_{ij}-f_j\delta_{ik}}-\pa{f_ku_{ij}-f_ju_{ik}} +u_i\pa{f_ku_j-f_ju_k}-\frac{1}{m-1}\pa{f_tu_{tk}\delta_{ij}-f_tu_{tj}\delta_{ik}} \\ \nonumber&+\frac{1}{m-1}\pa{f_tu_t}\pa{u_k\delta_{ij}-u_j\delta_{ik}}-\frac{1}{m-1}\abs{\nabla u}^2\pa{f_k\delta_{ij}-f_j\delta_{ik}}.
\end{align}

\begin{rem} A computation using equation \eqref{CGRS_comp_Ricci} shows that the tensor $D^{\pa{u, f}}$ can also be written as follows:
\begin{align}\label{tensorD_u_f}
  D^{\pa{u, f}}_{ijk} &=\frac{1}{(m-1)(m-2)}\sq{f_t\pa{f_{tj}\delta_{ik}-f_{tk}\delta_{ij}}-\abs{\nabla f}^2\pa{u_j\delta_{ik}-u_k\delta_{ij}}+\pa{f_tu_t}\pa{f_j\delta_{ik}-f_k\delta_{ij}}}\\ \nonumber &-\frac{1}{m-2}\sq{f_{ij}f_k-f_{ik}f_j+f_i\pa{u_kf_j-u_jf_k}}+\frac{\Delta f}{(m-1)(m-2)}\pa{f_k\delta_{ij}-f_j\delta_{ik}}.
\end{align}
\end{rem}
We have the following
\begin{proposition}\label{PR_CGRSDparticularcases}
  If the conformal gradient soliton is a conformal Einstein manifold (i.e. $f$ is constant) then $\left.D^{\pa{u, f}}\right|_{f=\text{const.}}\equiv 0$, while if the conformal gradient soliton is a soliton (i.e. $u=0$) then $\left.D^{\pa{u, f}}\right|_{u=0}=D^{\pa{0, f}}=D$.
\end{proposition}
\begin{proof} The proof is straightforward using the definition of $D^{\pa{u, f}}$ given in \eqref{tensorD_u_f_bestVersion}. Using instead the definition \eqref{tensorD_u_f}, for the first condition we just observe that the right hand side of \eqref{tensorD_u_f} vanishes when $f$ is constant. If $u=0$ then equation \eqref{tensorD_u_f} becomes
  \begin{equation}
  D^{\pa{0, f}}_{ijk}=\frac{1}{(m-1)(m-2)}\sq{f_t\pa{f_{tj}\delta_{ik}-f_{tk}\delta_{ij}}}-\frac{1}{m-2}\pa{f_{ij}f_k-f_{ik}f_j}+\frac{\Delta f}{(m-1)(m-2)}\pa{f_k\delta_{ij}-f_j\delta_{ik}}.
\end{equation}
Now the conclusion follows using the solitons equation \eqref{GradRicSolEqComponents} and its traced version $S+\Delta f = \lambda m$.

\end{proof}
From the definition \eqref{tensorD_u_f_bestVersion} of $D^{\pa{u, f}}$ and from equation \eqref{DExpComp} we immediately deduce the following
\begin{lemma} If $\varrg$ is a conformal gradient Ricci soliton then
  \begin{equation}\label{CGRS_D_ufvsTildeD}
    D^{\pa{u, f}} = e^{3u}\tilde{D}.
  \end{equation}
\end{lemma}

The first main result of this section is the following

\begin{theorem}\label{TH_FirstConditionCGRS}
If $\varrg$ is a conformal gradient Ricci soliton then
\begin{equation}\label{Eq_FirstCondition_CGRSCompNewD}
  C_{ijk}-\sq{(m-2)u_t-f_t}W_{tijk} = D^{\pa{u, f}}_{ijk}.
\end{equation}
\end{theorem}
  \begin{rem}
    Equation \eqref{Eq_FirstCondition_CGRSCompNewD} is the first integrability condition for a conformal gradient Ricci soliton. Moreovoer, using Proposition \ref{PR_CGRSDparticularcases}, when $f$ is constant we recover equation \eqref{FirstCond_GN} of Gover and Nurowski, while when $u=0$ we recover equation \eqref{firstCaoChen} of Cao and Chen.
  \end{rem}
\begin{proof}
  There are two ways to prove \eqref{Eq_FirstCondition_CGRSCompNewD}.

\emph{First proof} (the direct one).

 We start from \eqref{CGRS_comp_Schouten}. Taking the covariant derivative and skew-symmetryzing with respect to the second and third index we get
 \begin{align}\label{CGRS_THP_Eq1}
   C_{ijk} &- (m-2)u_tR_{tijk} + f_tR_{tijk} +(m-2)\pa{u_{ik}u_j-u_{ij}u_k}+ f_{ij}u_k-f_{ik}u_j +u_{ij}f_k-u_{ik}f_j= \\ \nonumber &+\frac{m-2}{m}\set{\sq{\frac{S_k}{2\pa{m-1}}-u_{ttk}+\frac{f_{ttk}}{m-2}}\delta_{ij}-\sq{\frac{S_j}{2\pa{m-1}}-u_{ttj}+\frac{f_{ttj}}{m-2}}\delta_{ik}} \\ \nonumber &+\frac{m-2}{m}\set{2u_t\sq{\pa{u_{tk}-\frac{f_{tk}}{m-2}}\delta_{ij}-\pa{u_{tj}-\frac{f_{tj}}{m-2}}\delta_{ik}}-\frac{2}{m-2}f_t\pa{u_{tk}\delta_{ij}-u_{tj}\delta_{ik}}}.
 \end{align}
 Tracing equation \eqref{CGRS_THP_Eq1} with respect to $i$ and $j$  we deduce the following interesting relation, which will come in handy later:
 \begin{align}\label{CGRS_SkUttkFttk}
 \frac{S_k}{2(m-1)}-u_{ttk}+\frac{f_{ttk}}{m-2} &=\frac{m}{m-1}\pa{u_tR_{tk}-\frac{1}{m-2}f_tR_{tk}}-\pa{\frac{m-2}{m-1}}u_uu_{tk}+\frac{1}{m-1}\pa{u_tf_{tk}+f_tu_{tk}}\\ \nonumber &-\frac{m}{m-1}\Delta u\,u_k + \frac{m}{(m-1)(m-2)}\pa{u_k\Delta f+f_k\Delta u}.
 \end{align}
 Substituting equation \eqref{CGRS_SkUttkFttk} in \eqref{CGRS_THP_Eq1}, using the definition of the Weyl tensor (see equation \eqref{Riemann_Weyl}) and rearranging we arrive at
\begin{align}\label{CGRS_THP_Eq2}
C_{ijk}-\sq{(m-2)u_t-f_t}W_{tijk} &= \frac{1}{(m-1)(m-2)}\sq{(m-2)u_t-f_t}\pa{R_{tj}\delta_{ik}-R_{tk}\delta_{ij}}+ \pa{R_{ik}u_j-R_{ij}u_k} \\ \nonumber &+\frac{1}{m-2}\pa{R_{ij}f_k-R_{ik}f_j} + \frac{S}{(m-1)(m-2)}\sq{(m-2)\pa{u_k\delta_{ij}-u_j\delta_{ik}}-\pa{f_k\delta_{ij}-f_j\delta_{ik}}} \\ \nonumber &+u_k\sq{(m-2)u_{ij}-f_{ij}}-u_j\sq{(m-2)u_{ik}-f_{ik}}+\pa{f_ju_{ik}-f_ku_{ij}}\\ \nonumber &+\pa{\frac{m-2}{m-1}}u_t\pa{u_{tk}\delta_{ij}-u_{tj}\delta_{ik}}+\frac{1}{m-1}u_t\pa{f_{tj}\delta_{ik}-f_{tk}\delta_{ij}}\\ \nonumber &+\frac{1}{m-1}f_t\pa{u_{tj}\delta_{ik}-u_{tk}\delta_{ij}}+\frac{1}{m-1}\sq{(m-2)\Delta u - \Delta f}\pa{u_j\delta_{ik}-u_k\delta_{ij}} \\ \nonumber &+\frac{1}{m-1}\Delta u\pa{f_k\delta_{ij}-f_j\delta_{ik}}.
\end{align}
Now we use \eqref{CGRS_comp_Ricci}  every time the Hessian of $u$ appears in equation \eqref{CGRS_THP_Eq2}; rearranging and simplifying (with a lot of patience) we deduce
\eqref{Eq_FirstCondition_CGRSCompNewD}.
\begin{rem}
The same argument obviously works in the case of conformally Einstein manifolds, leading to equation \eqref{FirstCond_GN}.
\end{rem}

 \emph{Second proof} (sketch). Since $\varrg$ is a conformal gradient Ricci soliton we have the validity of \eqref{firstCaoChen} with respect to the metric $\tilde{g}$, i.e.
 \[
 \tilde{C}_{ijk} + \tilde{f_t}\tilde{W}_{tijk} = \tilde{D}_{ijk};
 \]
 multiplying both members through $e^{3u}$ we get
 \[
 e^{3u}\tilde{C}_{ijk} + \pa{e^{u}\tilde{f_t}}\pa{e^{2u}\tilde{W}_{tijk}} = e^{3u}\tilde{D}_{ijk}.
 \]
Now using \eqref{Weylexp}, \eqref{Cottonlexp}, \eqref{DExpComp} and the fact that $e^{u}\tilde{f}_t=f_t$ we obtain \eqref{Eq_FirstCondition_CGRSCompNewD}.
\end{proof}
\begin{rem}
The first proof of Theorem \ref{TH_FirstConditionCGRS} is long but elementary, using only the definition of the Cotton tensor and the equation defining a conformal Ricci soliton. The second proof  is obviously shorter, but requires a lot of preliminary work to deduce the necessary transformation laws.
\end{rem}
As far as the second integrability condition is concerned we have
\begin{theorem}\label{CGRS_TH_SecCond}
  If $\varrg$ is a conformal gradient Ricci soliton then
  \begin{equation}\label{Eq_SecondConditionBach}
    B_{ij}= \frac{1}{m-2}\set{D^{\pa{u, f}}_{ijk, k} -\pa{\frac{m-3}{m-2}}\sq{\pa{m-2}u_t-f_t}C_{jit}+\sq{f_tu_k+f_ku_t-(m-2)u_tu_k}W_{itjk}};
  \end{equation}
  Equivalently,
  \begin{align}\label{Eq_SecondConditionBach_equivalent}
  B_{ij}=\frac{1}{m-2}&\left\{\sq{(m-2)(m-4)u_tu_k-(m-4)\pa{u_kf_t+f_ku_t}+\pa{\frac{m-3}{m-2}}f_tf_k}W_{itjk}\right.\\\nonumber &\left.-\pa{\frac{m-3}{m-2}}\sq{\pa{m-2}u_t-f_t}D^{\pa{u, f}}_{jit}+D^{\pa{u, f}}_{ijt, t}\right\}.
  \end{align}
  \end{theorem}
  \begin{rem}
     Equation \eqref{Eq_SecondConditionBach} is the second integrability condition for a conformal gradient Ricci soliton. Moreover, if $f$ is constant we recover equation \eqref{SecondCond_GN} of Gover and Nurowski, while if $u=0$ we recover equation \eqref{secondCaoChen} of Cao and Chen.
  \end{rem}
  \begin{proof}
    Again, there are two ways to prove \eqref{Eq_SecondConditionBach}.

     \emph{First proof} (the direct one).
We take the covariant derivative of equation \eqref{Eq_FirstCondition_CGRSCompNewD} to get
\begin{equation}\label{CGRS_SC_Eq1}
C_{ijk, l} -\sq{(m-2)u_{tl}-f_{tl}}W_{tijk}-\sq{(m-2)u_t-f_t}W_{tijk, l} = D^{\pa{u, f}}_{ijk, l};
\end{equation}
tracing with respect to $k$ and $l$ and using the definition of the Bach tensor and the fact that $W_{tijk, k} = W_{kjit, k} = -\pa{\frac{m-3}{m-2}}C_{jit}$ we deduce
\begin{equation}\label{CGRS_SC_Eq2}
(m-2)B_{ij}-\sq{R_{tk}-(m-2)u_{tk}+f_{tk}}W_{itjk} +\pa{\frac{m-3}{m-2}}\sq{(m-2)u_t-f_t}C_{jit}=D^{\pa{u, f}}_{ijk, k}.
\end{equation}
Now we note that, by equation \eqref{CGRS_comp_Ricci},
\[R_{tk}-(m-2)u_{tk}+f_{tk} = -(m-2)u_tu_k+f_tu_k+f_ku_t+\frac 1m\sq{S-(m-2)\pa{\Delta u-\abs{\nabla u}^{2}}+\Delta f-2g\pa{\nabla f, \nabla u}}\delta_{tk};
\]
substituting in \eqref{CGRS_SC_Eq2} and computing  we obtain \eqref{Eq_SecondConditionBach}. Equation \eqref{Eq_SecondConditionBach_equivalent} can be now obtained using \eqref{Eq_FirstCondition_CGRSCompNewD} in \eqref{Eq_SecondConditionBach} and rearranging.

\emph{Second proof} (sketch). Since $\varrg$ is a conformal gradient Ricci soliton we have the validity of \eqref{secondCaoChen} with respect to the metric $\tilde{g}$, i.e.
\[
(m-2)\tilde{B}_{ij} = \tilde{D}_{ijt, t} + \pa{\frac{m-3}{m-2}}\tilde{f}_t\tilde{C}_{jit};
\]
the thesis now follows from \eqref{BachExpComp}, \eqref{CovDerivDExpComp} (traced with respect to $k$ and $t$), \eqref{Cottonlexp} and a long computation.
  \end{proof}
\begin{rem} Following the second proof of Theorem \ref{CGRS_TH_SecCond} it is possibile to show that

    \begin{equation}
    D^{\pa{u, f}}_{ijt, t} = e^{2u}\tilde{D}_{ijt, t} - (m-4)u_tD^{\pa{u, f}}_{ijt} + u_tD^{\pa{u, f}}_{jit}.
    \end{equation}
\end{rem}

We observe that equation \eqref{CGRS_SkUttkFttk} gives a relation between $\nabla S$, $\nabla\Delta u$ and $\nabla \Delta f$ for a conformal gradient Ricci soliton. On the other hand, taking the covariant derivative of equation \eqref{Eq_CGRSGlobalTraced}, we deduce that
\begin{align}\label{CGRS_SkUttkFttk_second}
 \frac{S_k}{2(m-1)}-u_{ttk}+\frac{f_{ttk}}{2(m-1)} &=(m-2)u_tu_{tk}-\frac{m-2}{2(m-1)}f_tu_{tk}-\frac{m-2}{2(m-1)}u_tf_{tk}+\frac{S}{m-1}u_k -2\Delta u\, u_k\\ \nonumber &-(m-2)\abs{\nabla u}^2u_k+\frac{1}{m-1}\Delta f \,u_k+\pa{\frac{m-2}{m-1}}\pa{f_tu_t}u_k.
 \end{align}
 Subtracting \eqref{CGRS_SkUttkFttk_second} from \eqref{CGRS_SkUttkFttk} and rearranging we obtain
 \begin{align}\label{CGRS_Fttk_prelim}
 f_{ttk} &= 2(m-2)u_tR_{tk}-2f_tR_{tk}-2\pa{m-2}^2u_tu_{tk}+(m-2)u_tf_{tk}+(m-2)f_tu_{tk}+2\frac{(m-2)^2}{m}\Delta u\, u_k\\ \nonumber &-2\frac{(m-2)}{m}S u_k +2\frac{(m-1)(m-2)^2}{m}\abs{\nabla u}^2u_k +\frac{4}{m}\Delta f\,u_k +2\Delta u\,f_k-2\frac{(m-2)^2}{m}\pa{f_tu_t}u_k.
 \end{align}
Now using  equation \eqref{CGRS_comp_Ricci} to substitute every term containing the Hessian of $u$ and rearranging we deduce the following
\begin{proposition}\label{PR_CGRS_nablaDeltaF}
Let $\pa{M, g, f, \lambda}$ be a conformal gradient Ricci soliton; then we have
\begin{align}\label{CGRS_Fttk}
 f_{ttk} &= f_tf_{tk}-f_tR_{tk}-(m-2)u_tf_{tk} + \frac{(m-2)(2m-1)}{m}\abs{\nabla u}^2f_k+2\Delta f\,u_k +\pa{\frac{3m-2}{m}}\Delta u\,f_k \\ \nonumber &+(m-2)g\pa{\nabla f, \nabla u}u_k-\abs{\nabla f}^2u_k-\frac{\pa{S+\Delta f}}{m}f_k-\pa{\frac{m-2}{m}}g\pa{\nabla f, \nabla u}f_k.
\end{align}
\end{proposition}
Inserting now \eqref{CGRS_Fttk} into \eqref{CGRS_SkUttkFttk} and rearranging we obtain the following, interesting expression for $\nabla\Delta u$.
\begin{theorem}\label{TH_CGRS_nablaDeltaU}
Let $\pa{M, g, f, \lambda}$ be a conformal gradient Ricci soliton; then we have
\begin{align}\label{CGRS_Uttk}
 u_{ttk} &= \frac{S_k}{2(m-1)}-u_tR_{tk}-u_tf_{tk}+\frac{1}{m-1}f_tf_{tk}+\pa{\frac{m-2}{m}}\abs{\nabla u}^2u_k+\pa{\frac{m-2}{m}}g\pa{\nabla f, \nabla u}u_k\\ \nonumber &-\frac{S}{m(m-1)}\pa{u_k+f_k}+\pa{\frac{m+2}{m}}\Delta u\,u_k-\frac{1}{m-1}\abs{\nabla f}^2u_k+\frac 1m\Delta f\,u_k + \frac{2(m-1)}{m}\abs{\nabla u}^2f_k \\ \nonumber &-\frac{m-2}{m(m-1)}g\pa{\nabla f, \nabla u}f_k+\frac{2}{m}\Delta u\,f_k-\frac{1}{m(m-1)}\Delta f\,f_k.
 \end{align}

\end{theorem}

%

\section{Generic Ricci solitons: necessary conditions}\label{SecGenericRS}

In this section we construct, for a generic Ricci solitons$(M, g, X, \lambda)$, two integrability conditions which are a direct generalization of the ones in section \ref{sec5}, valid for a gradient Ricci solitons. To state them we first need to define the tensor $D^X$ as follows:
\begin{align}\label{DXtensor_components}
  D^X_{ijk} &= \frac{1}{m-2}\pa{X_kR_{ij}-X_jR_{ik}}+\frac{1}{(m-1)(m-2)}\pa{X_tR_{tk}\delta_{ij}-X_tR_{tj}\delta_{ik}}-\frac{S}{(m-1)(m-2)}\pa{X_k\delta_{ij}-X_j\delta_{ik}}\\ \nonumber &+\frac 12\pa{X_{kji}-X_{jki}}+\frac{1}{2(m-1)}\sq{\pa{X_{tkt}-X_{ktt}}\delta_{ij}-\pa{X_{tjt}-X_{jtt}}\delta_{ik}}.
\end{align}
\begin{rem}
  If $X=\nabla f$ for some $f\in \cinf$, then $D^{\nabla f} \equiv D$ (since $X_{kji}=f_{kji}=X_{jki}=f_{jki})$.
\end{rem}
The following theorem shows that $D^X$ is the natural counterpart of $D$ in the generic case:

\begin{theorem} \label{GenRS_TH_FirstCond}
If $\pa{M, g, X, \lambda}$ is a generic Ricci soliton with respect to the smooth vector field $X$, then the Cotton tensor, the Weyl tensor, the Bach tensor, $X$ and the tensor $D^X$ satisfy the conditions:
  \begin{align}
     &C_{ijk}+X_t W_{tijk} = D^X_{ijk}, \label{firstGenericRSIntCondition}\\     &B_{ij} = \frac{1}{m-2}\pa{D^X_{ijk, k}+\frac{m-3}{m-2}X_tC_{jit}+\frac 12\pa{X_{tk}-X_{kt}}W_{itjk}}. \label{secondGenericRSIntCondition}
  \end{align}
\end{theorem}

\begin{rem}
  If $X=\nabla f$ for some $f\in \cinf$, equations \eqref{firstGenericRSIntCondition} and \eqref{secondGenericRSIntCondition} become, respectively, \eqref{firstCaoChen} and \eqref{secondCaoChen}.
\end{rem}

\begin{rem} From \eqref{firstGenericRSIntCondition} we deduce
\begin{equation}
X_tC_{tij} = X_tD_{tij}.
\end{equation}
\end{rem}
We omit here the proof, since Theorem \ref{GenRS_TH_FirstCond} will be a consequence of Theorems \ref{TH_FirstConditionCGenericRS} and \ref{TH_SecondConditionCGenericRS} of the next section.

\section{Conformal generic Ricci solitons}\label{Sec_Conf_GenRS}
As a further step toward generalization, not unexpectedly, in this section we
define the notion of a conformal generic Ricci soliton.

\begin{defi}
  A Riemannian manifold $\varrg$ is said to be a \emph{conformal generic Ricci soliton} if there exist a conformal change of the metric  $\tilde{g}=e^{2u} g, \, u\in\cinf$, a smooth vector field $X \in \mathfrak{X}(M)$ and a constant $\lambda \in \erre$ such that $\pa{M, \tilde{g}}$ is a generic Ricci soliton, i.e.
  \begin{equation}\label{CGenericRS_global_tilde}
  \widetilde{\ricc} +\frac 12\mathcal{L}_X\tilde{g} = \lambda\tilde{g}.
  \end{equation}
\end{defi}

In terms of the geometry of the manifold $\varrg$,  \eqref{CGenericRS_global_tilde} leads to the following

\begin{lemma}
$\varrg$ is a conformal generic Ricci soliton if and only if there exist  $u\in\cinf$, a smooth vector field $X \in \mathfrak{X}(M)$ and a constant $\lambda \in \erre$ such that
\begin{align}\label{Eq_CGenericRSGlobal}
\ricc - (m-2)\hess\pa{u}+&(m-2)du\otimes du +\frac 12 e^{2u}\mathcal{L}_Xg = \frac{1}{m}\sq{S-(m-2)\pa{\Delta u-\abs{\nabla u}^{2}}+e^{2u}\diver X}g
\end{align}
and
\begin{align}\label{Eq_CGenericRSGlobalTraced}
S-2(m-1)\Delta u - (m-1)(m-2)\abs{\nabla u}^{2}+e^{2u}\pa{\diver X+m g\pa{X, \nabla u}} = m\lambda e^{2u}.
\end{align}

\end{lemma}

\begin{proof}
In an orthonormal frame \eqref{CGenericRS_global_tilde} becomes
 \begin{equation}\label{CGenericRS_comp_tilde}
   \tilde{R_{ij}}+\frac 12\pa{\tilde{X}_{ij}+\tilde{X}_{ji}}=\lambda \delta_{ij},
 \end{equation}
 while tracing \eqref{CGenericRS_global_tilde} we deduce that
  \begin{equation}\label{Eq_lambda_tildeS_tildeDeltaf_generic}
   m\lambda=\tilde{S} +\tilde{\diver}X.
 \end{equation}
 Multiplying both sides of \eqref{Eq_lambda_tildeS_tildeDeltaf_generic} by $e^{2u}$ and using \eqref{scalarExp} and \eqref{divergenzatilde} we get \eqref{Eq_CGenericRSGlobalTraced}; multiplying both sides of \eqref{CGenericRS_comp_tilde} by $e^{2u}$, using \eqref{RicciexpComponents}, \eqref{tildeXiktildeXki} and \eqref{Eq_CGenericRSGlobalTraced} we deduce
 \begin{equation}\label{CGenericRS_comp_Ricci}
R_{ij} -(m-2)u_{ij} + (m-2)u_iu_j +\frac 12 e^{2u}\pa{X_{ij}+X_{ji}} =\frac{1}{m}\sq{S-(m-2)\pa{\Delta u-\abs{\nabla u}^{2}}+e^{2u}\diver X}\delta_{ij},
\end{equation}
that is \eqref{Eq_CGenericRSGlobal}.
 \end{proof}

\begin{rem}
  If $u=0$ equations \eqref{Eq_CGenericRSGlobal} and \eqref{Eq_CGenericRSGlobalTraced} give
  \begin{equation}
    R_{ij} + \frac 12 \pa{X_{ij}+X_{ji}} = \frac 1m \pa{S+\diver X}\delta_{ij} = \lambda \delta_{ij},
  \end{equation}
  that is the equation of generic Ricci solitons;  if in addition $X=\nabla f$ for some $f\in \cinf$, we obviously recover the equation of gradient Ricci solitons. On the other hand, if $u \not \equiv 0$ but $X$ is the gradient of some function $f$ with the respect to the metric $\tilde g$, we recover equations \eqref{Eq_CGRSGlobal} and \eqref{Eq_CGRSGlobalTraced}. To prove this we observe that

  \[
  X = \tilde{\nabla}f = \tilde{f}_i\tilde{e}_i = \tilde{f}_ie^{-u}e_i = e^{-2u}f_ie_i,
  \]
  so we deduce
  \[
   X = \tilde{\nabla}f = e^{-2u}\nabla f.
  \]
  Moreover we have
  \begin{align*}
    X_i &= e^{-2u}f_i, \\ X_{ij} &= e^{-2u}\pa{f_{ij}-2f_iu_j}, \,\, X_{ji} = e^{-2u}\pa{f_{ij}-2f_ju_i}, \\ \diver X &= X_{ii} = e^{-2u}\pa{\Delta f -2g\pa{\nabla u, \nabla f}}.
  \end{align*}
  Substituting the previous relations in \eqref{CGenericRS_comp_Ricci} we get \eqref{CGRS_comp_Ricci}.
\end{rem}

Note that equation \eqref{CGenericRS_comp_Ricci} can be written, using the Schouten tensor, as
\begin{equation}\label{CGenericRS_comp_Schouten}
A_{ij} -(m-2)u_{ij} + (m-2)u_iu_j +\frac 12 e^{2u}\pa{X_{ij}+X_{ji}}=\frac{1}{m}\sq{\frac{m-2}{2(m-1)}S-(m-2)\pa{\Delta u-\abs{\nabla u}^{2}}+e^{2u}\diver X}\delta_{ij}.
\end{equation}

For a conformal generic Ricci soliton we now define the tensor $D^{\pa{u, X}}$ as follows:
\begin{align}\label{tensorD_u_X}
  D^{\pa{u, X}}_{ijk} =e^{2u}&\left\{\frac{1}{m-2}\pa{X_kR_{ij}-X_jR_{ik}}+\frac{1}{(m-1)(m-2)}\pa{X_tR_{tk}\delta_{ij}-X_tR_{tj}\delta_ik}-\frac{S}{(m-1)(m-2)}\pa{X_k\delta_{ij}-X_j\delta_{ik}}\right.\\ \nonumber &+\frac 12\pa{X_{kji}-X_{jki}}+\frac{1}{2(m-1)}\sq{\pa{X_{tkt}-X_{ktt}}\delta_{ij}-\pa{X_{tjt}-X_{jtt}}\delta_{ik}}-\frac 12\sq{\pa{X_{ij}+X_{ji}}u_k-\pa{X_{ik}+X_{ki}}u_j} \\ \nonumber &\left.-\frac{1}{2(m-1)}u_t\sq{\pa{X_{tk}+X_{kt}}\delta_{ij}-\pa{X_{tj}+X_{jt}}\delta_{ik}}+\frac{1}{m-1}\pa{\diver X} \pa{u_k\delta_{ij}-u_j\delta_{ik}}\right\}.
\end{align}

We have the following
\begin{proposition}
  If the conformal generic Ricci soliton is a conformal Einstein manifold (i.e. $X \equiv 0$) then $\left.D^{\pa{u, X}}\right|_{X \equiv 0}=D^{\pa{u, 0}}\equiv 0$, while if the conformal generic Ricci soliton is a generic Ricci soliton (i.e. $u=0$) then $\left.D^{\pa{u, X}}\right|_{u=0}=D^{\pa{0, X}}=D^X$.
\end{proposition}
\begin{proof}
  The proof is just a straightforward calculation.
\end{proof}
A computation similar to the one leading to equation \eqref{CGRS_D_ufvsTildeD} shows the validity of the following
\begin{lemma}
If $\varrg$ is a conformal generic Ricci soliton then
  \begin{equation}\label{CGeRS_EqDuXe3uDX}
    D^{\pa{u, X}} = e^{3u}\tilde{D}^X.
  \end{equation}
\end{lemma}

We now come to the main result of this section, i.e. the first integrability condition for conformal generic Ricci solitons.

\begin{theorem}\label{TH_FirstConditionCGenericRS}
If $\varrg$ is a conformal generic Ricci soliton then
\begin{align}\label{Eq_FirstCondition_CGenericRSComponents}
C_{ijk}-\sq{(m-2)u_t-e^{2u}X_t}W_{tijk} = D^{(u, X)}_{ijk}.
\end{align}
\end{theorem}
\begin{rem}
  If $u=0$, \eqref{Eq_FirstCondition_CGenericRSComponents} becomes equation \eqref{firstGenericRSIntCondition}; if $X=0$, we recover equation \eqref{FirstCond_GN}; if $X=\tilde{\nabla}f$ for some $f \in \cinf$, we have equation \eqref{Eq_FirstCondition_CGRSCompNewD}.
\end{rem}

\begin{proof}
  As in the case of Theorem  \ref{TH_FirstConditionCGRS}, there are two equivalent ways to prove \eqref{Eq_FirstCondition_CGenericRSComponents}.

 \emph{First proof} (the direct one).

 We start from \eqref{CGenericRS_comp_Schouten}. Taking the covariant derivative and skew-symmetryzing with respect to the second and third index we get
 \begin{align}\label{CGers_FirstCond_Eq1}
   &C_{ijk} - (m-2)u_tR_{tijk} +(m-2)\pa{u_{ik}u_j-u_{ij}u_k} \\ \nonumber &+ e^{2u}\sq{\pa{X_{ij}+X_{ji}}u_k-\pa{X_{ik}+X_{ki}}u_j}+\frac 12e^{2u}X_tR_{tijk}+\frac 12e^{2u}\pa{X_{jik}-X_{kji}} \\ \nonumber &-\frac{1}{m}\left\{\frac{(m-2)}{2\pa{m-1}}\pa{S_k\delta_{ij}-S_j\delta_{ik}}-(m-2)\pa{u_{ttk}\delta_{ij}-u_{ttj}\delta_{ik}}+e^{2u}\pa{X_{ttk}\delta_{ij}-X_{ttj}\delta_{ik}}\right.\\ \nonumber &\left.\quad+2(m-2)\pa{u_tu_{tk}\delta_{ij}-u_tu_{tj}\delta_{ik}}+2 e^{2u} (\diver X) \pa{u_k\delta_{ij}-u_j\delta_{ik}}\right\}=0.
 \end{align}
 Note that, using the first Bianchi identity \eqref{FirstBianchiRiem} and Lemma \ref{CommutationsForVectorFields}, we have
 \[
 X_{jik}-X_{kji}=X_{jki}-X_{kji}+X_tR_{tijk},
 \]
 so that
 \begin{equation}\label{CGeRS_commutXjik}
 \frac 12e^{2u}X_tR_{tijk}+\frac 12e^{2u}\pa{X_{jik}-X_{kji}} = e^{2u}X_tR_{tijk}+\frac 12 e^{2u}\pa{X_{jki}-X_{kji}}.
 \end{equation}
Tracing equation \eqref{CGers_FirstCond_Eq1} with respect to $i$ and $j$ we deduce the following interesting relation (compare it with equation \eqref{CGRS_SkUttkFttk}):
\begin{align}\label{CGeRS_SkUttkXttk}
 \frac{(m-2)}{2(m-1)}S_k-(m-2)u_{ttk}+e^{2u}X_{ttk} &=\frac{m}{m-1}\sq{(m-2)u_t-e^{2u}X_t}R_{tk}-\frac{(m-2)^2}{m-1}u_uu_{tk}+\frac{2}{m-1}e^{2u}\pa{\diver X}\,u_k\\ \nonumber &-\frac{m(m-2)}{m-1}\Delta u\,u_k -\frac{m}{m-1}e^{2u}u_t\pa{X_{tk}+X_{kt}}+\frac{m}{2(m-1)}e^{2u}\pa{X_{tkt}-X_{ktt}}.
 \end{align}
 Now we insert \eqref{CGeRS_SkUttkXttk}, \eqref{CGeRS_commutXjik} and \eqref{Riemann_Weyl} into \eqref{CGers_FirstCond_Eq1}; after some manipulation we arrive at
 \begin{align}\label{CGers_FirstCond_Eq2}
 C_{ijk}-\sq{(m-2)u_t-e^{2u}X_t}W_{tijk} &= \frac{1}{(m-1)(m-2)}\sq{(m-2)u_t-e^{2u}X_t}\pa{R_{tj}\delta_{ik}-R_{tk}\delta_{ij}}+ \\ \nonumber &+\frac{1}{m-2}R_{ik}\sq{(m-2)u_j-e^{2u}X_j}
 -\frac{1}{m-2}R_{ij}\sq{(m-2)u_k-e^{2u}X_k} \\ \nonumber &+\frac{S}{(m-1)(m-2)}\set{\sq{(m-2)u_k-e^{2u}X_k}\delta_{ij}-\sq{(m-2)u_j-e^{2u}X_j}\delta_{ik}} \\ \nonumber &+(m-2)\pa{u_{ij}u_k-u_{ik}u_j}+e^{2u}\sq{\pa{X_{ik}+X_{ki}}u_j-\pa{X_{ij}+X_{ji}}u_k} \\ \nonumber &+\frac 12 e^{2u}\pa{X_{kji}-X_{jki}}+\pa{\frac{m-2}{m-1}}u_t\pa{u_{tk}\delta_{ij}-u_{tj}\delta_{ik}}\\ \nonumber &+\frac{2}{m-1}e^{2u}\pa{\diver X}\pa{u_k\delta_{ij}-u_j\delta_{ik}}-\pa{\frac{m-2}{m-1}}\Delta u\pa{u_k\delta_{ij}-u_j\delta_{ik}}\\ \nonumber &-\frac{1}{m-1}e^{2u}u_t\sq{\pa{X_{tk}+X_{kt}}\delta_{ij}-\pa{X_{tj}+X_{jt}}\delta_{ik}} \\ \nonumber &+\frac{1}{2(m-1)}e^{2u}\sq{\pa{X_{tkt}-X_{ktt}}\delta_{ij}-\pa{X_{tjt}-X_{jtt}}\delta_{ik}}.
 \end{align}
Using \eqref{CGenericRS_comp_Ricci} every time the Hessian of $u$ appears in equation \eqref{CGers_FirstCond_Eq2}, rearranging and simplifying (with a lot of patience, again) we deduce \eqref{Eq_FirstCondition_CGenericRSComponents}.

 \emph{Second proof (sketch)}. Since $\varrg$ is a conformal generic Ricci soliton we have the validity of \eqref{firstGenericRSIntCondition} with respect to the metric $\tilde{g}$, i.e.
 \[
 \tilde{C}_{ijk} + \tilde{X_t}\tilde{W}_{tijk} = \tilde{D}^X_{ijk}.
 \]
Now one should multiply both members through $e^{3u}$, use \eqref{Weylexp}, \eqref{Cottonlexp}, the fact that $\tilde{X}_t=e^{u}X_t$ and the computation producing equation  \eqref{CGeRS_EqDuXe3uDX}.
\end{proof}

As far as the second integrability condition is concerned we have

\begin{theorem}\label{TH_SecondConditionCGenericRS}
  If $\varrg$ is a conformal generic Ricci soliton then
  \begin{align}\label{Eq_SecondConditionBach_GENERIC}
  B_{ij}= \frac{1}{m-2}\set{D^{\pa{u, X}}_{ijk, k}-\pa{\frac{m-3}{m-2}}\sq{(m-2)u_t-e^{2u}X_t}C_{jit}+\sq{\frac 12e^{2u}\pa{X_{tk}-X_{kt}}+2e^{2u}\pa{X_tu_k}-(m-2)u_tu_k}W_{itjk}}.
  \end{align}
  \end{theorem}
\begin{rem}
  If $u=0$, \eqref{Eq_SecondConditionBach_GENERIC} becomes equation \eqref{secondGenericRSIntCondition}; if $X=0$, we recover equation \eqref{SecondCond_GN}; if $X=\tilde{\nabla}f$ for some $f \in \cinf$, we have equation \eqref{Eq_SecondConditionBach}.
\end{rem}
\begin{proof}
  Taking the covariant derivative of equation \eqref{Eq_FirstCondition_CGenericRSComponents} we get
  \[
  C_{ijk, l} -\sq{(m-2)u_{tl}-2e^{2u}X_tu_l-e^{2u}X_{tl}}W_{tijk}-\sq{(m-2)u_t-e^{2u}X_t}W_{tijk, l} = D^{\pa{u, X}}_{ijk, l}.
  \]
  Now we trace with respect to $k$ and $l$ and we use the definition \eqref{def_Bach_comp} of the Bach tensor to deduce
  \[
  (m-2)B_{ij}-R_{tk}W_{itjk}+\sq{(m-2)u_{tk}-2e^{2u}X_tu_k-e^{2u}X_{tk}}W_{itjk}-\sq{(m-2)u_t-e^{2u}X_t}W_{tijk, k} = D^{\pa{u, X}}_{ijk, k}.
  \]
  Inserting \eqref{def_Cotton_comp_Weyl} and \eqref{CGenericRS_comp_Ricci} in the previous relation, simplifying and rearranging we get \eqref{Eq_SecondConditionBach_GENERIC}.
\end{proof}


\vspace{3cm}

\section{Higher order integrability condition for gradient Ricci solitons}\label{sec_higherorder}
In this short section we present the third and the fourth integrability conditions for gradient Ricci solitons of dimension $m\geq 4$. Starting from equation \eqref{secondCaoChen} in Theorem \ref{caochenth} we get the following
\begin{theorem}\label{TH_thirdIntCond}
  If $\pa{M, g, f}$ is a gradient Ricci soliton with potential function $f$, then the Cotton tensor,  the Bach tensor and the tensor $D$ satisfy the condition
  \begin{equation}\label{thirdCond1}
  R_{kt}C_{kti}= (m-2)D_{itk, tk},
\end{equation}
or, equivalently,
\begin{equation}\label{thirdCond2}
  \pa{\diver B}_i = B_{ik, k} = \pa{\frac{m-4}{m-2}}D_{itk, tk}.
\end{equation}
\end{theorem}
\begin{proof}
  We take the covariant derivative of equation \eqref{secondCaoChen}, obtaining
  \[
  (m-2)B_{ij, k} = D_{ijt, tk} + \pa{\frac{m-3}{m-2}}\pa{f_{tk}C_{jit}+f_tC_{jit, k}},
  \]
  which implies, using the soliton equation,
  \[
    (m-2)B_{ij, k} = D_{ijt, tk} + \pa{\frac{m-3}{m-2}}\pa{\lambda C_{jik}+ R_{tk}C_{jti} + f_tC_{jit, k}}.
  \]
  Tracing with respect to $j$ and $k$, using equation \eqref{NullDiverCotton} and the fact that the Cotton tensor is totally trace-free we get
  \[
   (m-2)B_{ik, k} = D_{ikt, tk} + \pa{\frac{m-3}{m-2}}R_{tk}C_{jti}.
  \]
  Now we exploit \eqref{diverBach} in the previous relation, obtaining \eqref{thirdCond1}. To get \eqref{thirdCond2} we simply insert again \eqref{diverBach} into \eqref{thirdCond1}.
\end{proof}
\begin{theorem}\label{TH_fourthIntCond}
  If $\pa{M, g, f}$ is a gradient Ricci soliton with potential function $f$, then the Cotton tensor,  the Bach tensor and the tensor $D$ satisfy the condition
  \begin{equation}\label{fourthCond1}
  \frac{1}{2}\abs{C}^2+\pa{m-2}R_{ij}B_{ij}-R_{ij}R_{kt}W_{ikjt}= (m-2)D_{itk, tki},
\end{equation}
or, equivalently,
\begin{equation}\label{fourthCond2}
  B_{ik, ki} = \pa{\frac{m-4}{m-2}}D_{itk, tki}.
\end{equation}
\end{theorem}
\begin{proof}
Equation \eqref{fourthCond2} follows by taking the divergence of \eqref{thirdCond2}. To get \eqref{fourthCond1} we take the divergence of \eqref{thirdCond1},
\[
R_{kt, i}C_{kti}+R_{kt}C_{kti, i} = (m-2)D_{itk, tki}.
\]
Now we use the symmetry of the Cotton tensor and the definition of the Bach tensor, obtaining
\[
\frac{1}{2}\pa{R_{kt, i}-R_{ki, t}}C_{kti}+R_{kt}\sq{(m-2)B_{kt}-R_{ij}W_{ikjt}} = (m-2)D_{itk, tki},
\]
from which we immediately deduce \eqref{fourthCond1}.
\end{proof}

\section{Open questions}\label{SecOpen}
We conclude the paper with a brief overview of interesting open problems.

First of all, \emph{sufficient conditions} for a generic Riemannian manifold to be conformally equivalent (locally or globally) to a Eistein manifold have been found
by several authors, see for instance Gover-Nurowsky \cite{GoverNurowski} and Listing \cite{Listing1}, \cite{Listing2}; it would be of great interest to find similar results in the Ricci soliton case.

Another interesting result would be to deduce some \emph{a priori} estimate on scalar curvature for conformally Einstein manifolds or conformally Ricci solitons,
using PDE methods to study scalar equations obtained from their structure; a similar approach has been used for instance in \cite{MasRigRim} and \cite{CMMRGen}.

In the spirit of \cite{CaoChen} and \cite{CaoCatinoChenMantMazz}, rigidity and classification results for Bach-flat gradient Ricci solitons can be derived using the first and
the second integrability conditions, see also \cite{CMMRGet}. It is then natural to ask if it is possible to obtain similar results under weaker assumptions, such as
$\diver B=0$, exploiting also the third and fourth integrability conditions provided in the previous section. We explicitly remark that in dimension three the condition
$\diver B=0$ is sufficient to obtain the classification, see \cite{CaoCatinoChenMantMazz}. Moreover, in obtaining the aforementioned classification results, a key role is played by the vanishing of the tensor $D$; it would be significant to identify weaker requirements
on the Bach tensor and/or its divergence that could ensure this condition.

%
%

\bibliographystyle{plain}

\bibliography{bibConfRicciSol_FEB14}

\def\cprime{$'$}
\begin{thebibliography}{10}

\bibitem{Bach}
R.~Bach.
\newblock Zur {W}eylschen {R}elativit\"atstheorie und der {W}eylschen
  {E}rweiterung des {K}r\"ummungstensorbegriffs.
\newblock {\em Math. Z.}, 9(1-2):110--135, 1921.

\bibitem{Besse}
A.~Besse.
\newblock {\em Einstein manifolds. {R}eprint of the 1997 edition}.
\newblock Classics in Mathematics. Springer-Verlag, Berlin, 2008.

\bibitem{brendle}
S.~Brendle.
\newblock Rotational symmetry of self-similar solutions to the {R}icci flow.
\newblock {\em Invent. Math.}, 194(3):731--764, 2013.

\bibitem{Brink}
H.~W. Brinkmann.
\newblock Riemann spaces conformal to {E}instein spaces.
\newblock {\em Math. Ann.}, 91(3-4):269--278, 1924.

\bibitem{CaoCatinoChenMantMazz}
H.-D. Cao, G.~Catino, Q.~Chen, C.~Mantegazza, and L.~Mazzieri.
\newblock Bach-flat gradient steady {R}icci solitons.
\newblock {\em arXiv:1107.4591v2 [math.DG]}, 2011.

\bibitem{HDCaoChen_Steady}
H.-D. Cao and Q.~Chen.
\newblock On locally conformally flat gradient steady {R}icci solitons.
\newblock {\em Trans. Amer. Math. Soc.}, 364(5):2377--2391, 2012.

\bibitem{CaoChen0}
H.-D. Cao and Q.~Chen.
\newblock On locally conformally flat gradient steady {R}icci solitons.
\newblock {\em Trans. Amer. Math. Soc.}, 364(5):2377--2391, 2012.

\bibitem{CaoChen}
H.-D. Cao and Q.~Chen.
\newblock On {B}ach-flat gradient shrinking {R}icci solitons.
\newblock {\em Duke Math. J.}, 162(6):1149--1169, 2013.

\bibitem{CaoZhou}
H.-D. Cao and D.~Zhou.
\newblock On complete gradient shrinking {R}icci solitons.
\newblock {\em J. Differential Geom.}, 85(2):175--185, 2010.

\bibitem{XDCaoWangZhang_ShrinkingRS}
X.~Cao, B.~Wang, and Z.~Zhang.
\newblock On locally conformally flat gradient shrinking {R}icci solitons.
\newblock {\em Commun. Contemp. Math.}, 13(2):269--282, 2011.

\bibitem{Catino_pinched}
G.~Catino.
\newblock Complete gradient shrinking {R}icci solitons with pinched curvature.
\newblock {\em Math. Ann.}, 355(2):629--635, 2013.

\bibitem{CatMantEv}
G.~Catino and C.~Mantegazza.
\newblock The evolution of the {W}eyl tensor under the {R}icci flow.
\newblock {\em Ann. Inst. Fourier (Grenoble)}, 61(4):1407--1435 (2012), 2011.

\bibitem{CMMRGen}
G.~Catino, P.~Mastrolia, D.~D. Monticelli, and M.~Rigoli.
\newblock Analytic and geometric properties of generic {R}icci solitons.
\newblock {\em arXiv:1403.6298v1 [math.DG], submitted.}, 2014.

\bibitem{CMMRGet}
G.~Catino, P.~Mastrolia, D.~D. Monticelli, and M.~Rigoli.
\newblock On the geometry of gradient {E}instein-type manifolds.
\newblock {\em arXiv:1402.3453v1 [math.DG], submitted.}, 2014.

\bibitem{DerdMasch}
A.~Derdzinski and G.~Maschler.
\newblock A moduli curve for compact conformally-{E}instein {K}\"ahler
  manifolds.
\newblock {\em Compos. Math.}, 141(4):1029--1080, 2005.

\bibitem{ELnM}
M.~Eminenti, G.~La Nave, and C.~Mantegazza.
\newblock Ricci solitons: the equation point of view.
\newblock {\em Manuscripta Math.}, 127:345--367, 2008.

\bibitem{GoverNurowski}
A.~R. Gover and P.~Nurowski.
\newblock Obstructions to conformally {E}instein metrics in {$n$} dimensions.
\newblock {\em J. Geom. Phys.}, 56(3):450--484, 2006.

\bibitem{hamilton}
R.S. Hamilton.
\newblock {\em The {R}icci flow on surfaces. \emph{Mathematics and general
  relativity (Santa Cruz,CA, 1986)}}, volume~71 of {\em Contemp. Math.}, pages
  237--262.
\newblock Am. Math. Soc., 1988.

\bibitem{Ivey}
T.~Ivey.
\newblock Ricci solitons on compact three-manifolds.
\newblock {\em Differential Geom. Appl.}, 3(4):301--307, 1993.

\bibitem{Jensen}
G.~R. Jensen.
\newblock Einstein metrics on principal fibre bundles.
\newblock {\em J. Differential Geometry}, 8:599--614, 1973.

\bibitem{KapadiaSparling}
D.~Kapadia and G.~Sparling.
\newblock A class of conformally {E}instein metrics.
\newblock {\em Classical Quantum Gravity}, 17(22):4765--4776, 2000.

\bibitem{Listing1}
M.~Listing.
\newblock Conformal {E}instein spaces in {$N$}-dimensions.
\newblock {\em Ann. Global Anal. Geom.}, 20(2):183--197, 2001.

\bibitem{Listing2}
M.~Listing.
\newblock Conformal {E}instein spaces in {$N$}-dimensions. {II}.
\newblock {\em J. Geom. Phys.}, 56(3):386--404, 2006.

\bibitem{MasMonRig_Curvature}
P.~Mastrolia, D.~D. Monticelli, and M.~Rigoli.
\newblock A note on curvature of {R}iemannian manifolds.
\newblock {\em J. Math. Anal. Appl.}, 399(2):505--513, 2013.

\bibitem{MasRigRim}
P.~Mastrolia, M.~Rigoli, and M.~Rimoldi.
\newblock Some {G}eometric {A}nalysis on {G}eneric {R}icci {S}olitons.
\newblock {\em Comm. Contemp. Math.}, 15(03), 2013.

\bibitem{MasRigSet}
P.~Mastrolia, M.~Rigoli, and A.G. Setti.
\newblock {\em Yamabe-type equations on complete, noncompact manifolds}, volume
  302 of {\em Progress in Mathematics}.
\newblock Birkh\"auser Verlag, Basel, 2012.

\bibitem{Naber}
A.~Naber.
\newblock Noncompact shrinking four solitons with nonnegative curvature.
\newblock {\em J. Reine Angew. Math.}, 645:125--153, 2010.

\bibitem{NiWallach}
L.~Ni and N.~Wallach.
\newblock On a classification of gradient shrinking solitons.
\newblock {\em Math. Res. Lett.}, 15(5):941--955, 2008.

\bibitem{perelman1}
G.~Perelman.
\newblock The entropy formula for the {R}icci flow and its geometric
  applications.
\newblock {\em arXiv:math/0211159v1 [math.DG]}, 2002.

\bibitem{perelman}
G.~Perelman.
\newblock Ricci flow with surgery on three manifolds.
\newblock {\em arXiv:math/0303109v1 [math.DG]}, 2003.

\bibitem{petwylie3}
P.~Petersen and W.~Wylie.
\newblock On the classification of gradient {R}icci solitons.
\newblock {\em Geom. Topol.}, 14(4):2277--2300, 2010.

\bibitem{PRiS}
S.~Pigola, M.~Rimoldi, and A.~G. Setti.
\newblock Remarks on non-compact gradient {R}icci solitons.
\newblock {\em Math. Z.}, 268(3-4):777--790, 2011.

\bibitem{WangZiller1}
M.~Y. Wang and W.~Ziller.
\newblock Existence and nonexistence of homogeneous {E}instein metrics.
\newblock {\em Invent. Math.}, 84(1):177--194, 1986.

\bibitem{WangZiller2}
M.~Y. Wang and W.~Ziller.
\newblock Einstein metrics on principal torus bundles.
\newblock {\em J. Differential Geom.}, 31(1):215--248, 1990.

\bibitem{YanoNagano}
K.~Yano and T.~Nagano.
\newblock Einstein spaces admitting a one-parameter group of conformal
  transformations.
\newblock {\em Ann. of Math. (2)}, 69:451--461, 1959.

\bibitem{ZHZhang}
Z.-H. Zhang.
\newblock On the completeness of gradient {R}icci solitons.
\newblock {\em Proc. Amer. Math. Soc.}, 137:2755--2759, 2009.

\end{thebibliography}
\end{document}